\documentclass[12pt]{amsart}
\usepackage{amssymb,amscd,amsmath}
\usepackage[makeroom]{cancel}
\usepackage{tikz}
\usepackage[all]{xy} 
\usepackage{enumerate}
\usepackage{lmodern,bm}

\newtheorem{fact}{Fact}[section]
\newtheorem{formula}[fact]{Formula}
\newtheorem{theorem}[fact]{Theorem}
\newtheorem{proposition}[fact]{Proposition}
\newtheorem{lemma}[fact]{Lemma}
\newtheorem{corollary}[fact]{Corollary}
\newtheorem{remark}[fact]{Remark}
\newtheorem{definition}[fact]{Definition}
\newtheorem{example}[fact]{Example}

\def\Sym{{\rm Sym}}
\def\GGG{{\mathcal G}}
\def\QQ{{\mathbb Q}}
\def\OO{{\mathcal O}}
\def\ZZ{{\mathbb Z}}
\def\CC{{\mathbb C}}
\def\PP{{\mathbb P}}
\def\AA{{\mathbb A}}
\def\NN{\mathbb N}
\def\Hom{{\rm Hom}}
\def\End{{\rm End}}
\def\T{{\bf T}}
\def\t{{\mathfrak t}}
\def\gr{Gro\-then\-dieck }

\def\EE{{\mathcal E}}
\def\vt{\vartheta}
\def\f{f}
\def\om{\omega}
\def\omm{{\underline \om}}
\def\TFn{T\!\F(n)}
\newcommand{\oomega}{{\boldsymbol \omega}}

\def\Ellh{\mathcal E\mkern-3.7mu\ell\mkern-2mu\ell}
\def\ELL{\mathcal{E\!L\!L}}
\def\ELLt{\widehat{\ELL}}
\def\Ellt{\widehat{\Ellh}}
\def\Ee{{E^{\rm ell}}}
\def\Ek{{E^{\rm K}}}
\def\Eh{E^{\rm coh}}
\def\Em{\mathfrak E} 
\def\eell{{e^{\rm ell}}}
\def\eekk{{e^{\rm K}}}
\def\eehh{e^{\rm coh}}

\def\vme{{\!\scriptscriptstyle\vee}}
\def\eee{{\rm e}}
\def\delt{\delta\!\big(}

\DeclareMathOperator\id{id}
\DeclareMathOperator\res{res}
\DeclareMathOperator\rank{rank}
\DeclareMathOperator\RR{R}
\DeclareMathOperator\TTT{\mathcal T}
\DeclareMathOperator\HH{\mathbb H}

\DeclareMathOperator\Fl{\mathcal F}

\DeclareMathOperator\C{\mathbb C}
\DeclareMathOperator{\GL}{GL}
\DeclareMathOperator\F{\mathcal F}
\DeclareMathOperator\Z{\mathbb Z}

\DeclareMathOperator\characters{characters}

\DeclareMathOperator\ww{\mathbf w}
\DeclareMathOperator\wwh{\widehat{\mathbf{w}}}
\DeclareMathOperator\hC{\mathfrak{C}}

\newcommand{\bb}[1]{\overline{#1}}
\newcommand{\LL}[1]{\mathcal{L}_{#1}}

\title{Elliptic classes of Schubert varieties via Bott-Samelson resolution}

\author{Rich\'ard Rim\'anyi}
\address{Department of Mathematics, University of North Carolina at Chapel Hill, USA}
\email{rimanyi@email.unc.edu}

\author{Andrzej Weber}
\address{Institute of Mathematics, University of Warsaw, Poland}
\email{aweber@mimuw.edu.pl}

\thanks{R.R. is supported by Simons Foundation Grant 523882. A.W. is supported by NCN grant 2013/\-08/\-A/\-ST1/\-00804 and 2016/23/G/ST1/04282 (Beethoven 2)}

\begin{document}

\begin{abstract} 
Based on recent advances on the relation between geometry and representation theory, we propose a new approach to elliptic Schubert calculus.
We study the equi\-va\-riant elliptic characteristic classes of Schubert varieties of the generalized full flag variety $G/B$. For this first we need to twist the notion of elliptic characteristic class of Borisov-Libgober by a line bundle, and thus allow the elliptic classes to depend on extra variables. Using the Bott-Samelson resolution of Schubert varieties we prove a BGG-type recursion for the elliptic classes, and study the Hecke algebra of our elliptic BGG operators. For $G=\GL_n(\C)$ we find representatives of the elliptic classes of Schubert varieties in natural presentations of the K theory ring of $G/B$, and identify them with the Tarasov-Varchenko weight function. 
As a byproduct we find another recursion, different from the known R-matrix recursion for the fixed point restrictions of weight functions. On the other hand the R-matrix recursion generalizes for arbitrary reductive group $G$.
\end{abstract}

\maketitle



\section{Introduction}

The study of Schubert varieties and their singularities is a field where topology, algebraic geometry, and representation theory meet. An effective strategy to study Schubert varieties is assigning characteristic classes to them. However, the most natural characteristic class, the fundamental class, does not exist beyond K theory \cite{BE}---that is, e.g. the elliptic fundamental class depends on choices (e.g. choice of resolution, or choice of basis in a Hecke algebra) \cite{LZ}. Important deformations of the fundamental class appeared recently in the center stage in cohomology and K theory, under the names of Chern-Schwartz-MacPherson class and Motivic Chern class, partially due to their relation to Okounkov's stable envelope classes. In this paper we study the elliptic analogue of the CSM and MC classes, the {\em elliptic characteristic class of Schubert varieties}, which unlike the fundamental class, does {\em not} depend on choices.

\subsection{Background, early history}

The development of the concept of elliptic genus and its underlying characteristic class (for smooth manifolds) started in the second half of the 1980s, see e.g.~\cite{Ochanine, Landweber, Hirzebruch, Witten, Krichever, Hohn,Totaro}. Already at the beginning, the $S^1$ or $\C^*$ equivariant versions were considered, 
see e.g. the surveys of the subject in the introductions to \cite{BoLi0} and \cite{Li}.

According to the cited works  the elliptic class of a {\em smooth} complex manifold $X$, in {cohomology} has the form 
 \[
 \Ellh(X)=td(X)\,ch(\ELL(TX))\in H^*(X;\CC)[[q,z]],
 \]
where $\ELL(TX)\in K_\T(X)[y^{\pm 1/2}][[q]]$ is the elliptic complex defined in \cite[Lemma 2.5.2]{Hohn}, \cite[\S3]{Totaro}, \cite[Formula (3)]{BoLi0}, with $y=\eee^z$, $-\eee^z$ or $\eee^{\pm 2\pi i z}$ depending on the author, c.f. Definition \ref{def:Ell0} below. Here $ch$ denotes the Chern character and $td$ is the Todd class. For convenience
we will use cohomology with $\C$ coefficients, and will not indicate it anymore. 

It is worth noting that the elliptic class above can be interpreted as the Hirzebruch class of the free loop space $\mathcal LX=$Map$(S^1,X)$ localized at the set of constant loops $X\subset \mathcal LX$. This idea comes from Witten \cite[\S5]{Witten} and it is well explained by a heuristic argument in \cite[\S7.4]{HirzebruchBook}. 

The formula above for $\Ellh$ holds in the equivariant situation, say with the torus $\T=(\C^*)^n$ action. The only change is that the Borel cohomology has to be completed with respect to the gradation, since the Chern character is given by an infinite series.


\subsection{Extension to singular varieties: Borisov-Libgober} 
Applications of the elliptic genus to mirror symmetry created a necessity to define the elliptic class for singular varieties.
Borisov and Libgober in \cite{BoLi1} constructed a modification of the elliptic class which can be applied not only to singular varieties but to singular pairs consisting of a variety together with a Weil divisor. Their elliptic class $\Ellt(X,D)$ is defined if the pair $(X,D)$ has at worst  {\it Kawamata log-terminal singularities} (KLT).
 It is convenient to embed the singular pair in a smooth ambient space and consider the elliptic class as an element of the equivariant cohomology of the ambient space.
 
The starting point to define elliptic classes for a Schubert variety   $X_\om\subset G/B$ is Theorem~\ref{fullflagcnonical}, which collects results about the canonical divisor from \cite{Ram1, Ram2, BrKu, GrKu}.
In particular it follows that $K_{X_\om}+\partial X_{\om}$ is a Cartier divisor. This means that
 the pair $(X_\om,\partial X_\om)$ is Gorenstein,  however it is generally not KLT.

\subsection{The elliptic class for Schubert varieties $\Ek(X_\om,\lambda)$} \label{13}
To overcome the non-KLT property for Schubert varieties, we perturb the boundary divisor by a fractional line bundle $\LL{\lambda}$ depending on the new weight-parameter $\lambda\in\t^*$. Thus we will obtain a KLT pair, to which the Borisov-Libgober construction can be applied:
 \[
\Eh(X_\om,\lambda):=\Ellt(X_\om,\partial X_\om-\LL{\lambda})\in   H^*_\T(G/B ) ^{_\wedge} ((z))[[q]],
\]
where  $^\wedge$  denotes the completion  with respect to the gradation. The superscript {\it coh} indicates that the class lives in the {\em cohomology} of the flag variety. It turns out that the dependence on $\lambda$  of the resulting elliptic class is meromorphic, with poles at the walls of the Weyl alcoves. 

We will see that $\Eh(X_\om,\lambda)$ is of the form 
$$\Eh(X_\om,\lambda)=td(G/B)\,ch(\Ek(X_\om,\lambda))$$
with
$$\Ek(X_\om,\lambda)\in \big(K_\T(G/B)(\eee^{z/N})\otimes \CC(\T^*)\big)[[q]],$$
where  $N$ is an integer such that the weight  $N\lambda$ is integral.
 The class $\Ek(X_\om,\lambda)$ living now in {\em K theory} is the main protagonists of this paper. 
Treating $\Ek(X_\omega,\lambda)$ as a function on $(\lambda,z)\in \t^*\times \CC$, by quasi-periodicity properties of the theta function
we arrive to 
$$\Ek(X_\om)\in \big(K_\T(G/B)\otimes \CC(\T^*\times \C^*)\big)[[q]],$$
where
$\T^*\simeq\t^*/\t^*_\ZZ$ is the dual torus, the quotient of $\t^*$ by the integral weights lattice.
After choosing appropriate variables, we will obtain 
fairly concrete recursions and expressions for $\Ek(X_\om)$ in terms of Jacobi theta functions.


\subsection{Cohomology vs K theory vs elliptic cohomology} \label{14}
The functor $\hat H^*_\T(-)[[q]]$ can be considered as an equivariant complex-oriented cohomology whose Euler class of a line bundle is the theta function $\theta$. For our purposes this theory serves as the equivariant elliptic cohomology. The elliptic class in elliptic cohomology has a short and natural definition. For a smooth variety $X$ with a torus $\T$ action the elliptic class is defined as the equivariant Euler class of the tangent bundle. Here we consider the extended torus $\T\times\CC^*$, where $\CC^*$ acts trivially on $X$ and with the scalar multiplication by inverses on the fibers of the bundle. In terms of Chern roots $\xi_1,\xi_2,\dots,\xi_n$  the elliptic class is given by the formula 
$$\Ellh^{\,\rm ell}(X)=\prod_{k=1}^n\theta(\xi_k-z)\,,$$
where $z$ is a variable corresponding to the $\CC^*$ factor.
The elliptic classes for two complex orientations in Borel cohomology are related by the formula 
$$\Ellh(X)=\frac{\eehh(X)}{\eell(X)}\Ellh^{\,\rm ell}(X)=\prod_{k=1}^n\xi_k\frac{\theta(\xi_k-z)}{\theta(\xi_k)},$$
where $\eehh$ and $\eell$ are the Euler classes in  cohomology and in the elliptic theory. 
This is a classical approach to generalized  cohomology theories, and passage from one elliptic class to another is a Riemann-Roch type transformation, see \cite[42.1.D]{FoFu}.  We extend this method to define the elliptic class of singular pairs in elliptic cohomology. An advantage of this is that the classes $\Ee(X_\om)$ for Schubert varieties have better transformation properties. Moreover, it will be convenient to study the quotient (``local class'')
\[
\frac{\Eh(X_\om)}{\eehh(TG/B)}=\frac{\Ee(X_\om)}{\eell(TG/B)}=ch\left(\frac{\Ek(X_\om)}{\eekk(TG/B)}\right),
\]
instead of the three numerators.

\subsection{The Grojnowski model}
We would like to note that our version of elliptic cohomology has rational or complex coefficients, therefore the essential problems described in \cite{Landweber} are omitted. We work with equivariant Borel-type theory.
Our results apply to the delocalized equivariant elliptic cohomology in the Grojnowski sense as well. In his sketch \cite{Grojnowski} Grojnowski suggests that the elliptic cohomology should be defined as a sheaf of algebras over a product of elliptic curves. 
It would contain  information about equivariant cohomology of all possible fixed points with respect to subtori. 
In this approach the elliptic cohomology class would be a section of that sheaf. For flag varieties the restriction map $H^*_\T(G/B)\to H^*_\T((G/B)^\T)$ is injective, and the inversion of the Euler class of the tangent bundle does not weaken the formulas. 

In another approach (see \cite{GKV, Ganter}), where the elliptic cohomology is a scheme, the elliptic class would be a section of a so-called {\it Thom sheaf} over the scheme ${\bf E}(G/B)$, see \cite[\S7.2]{Ganter}. 
In this approach the restriction to fixed points corresponds to passing to the disjoint union of the products of elliptic curves.

These constructions of equivariant elliptic cohomology theories are not relevant for us. Our objective is to describe the combinatorics governing the characteristic classes, which can be achieved for the notions in Sections \ref{13}, \ref{14} above.

\subsection{Recursions}
The fixed points $(G/B)^\T$ are identified with the elements of the Weyl group~$W$. 
The local class  $\Ek(X_\om)/\eekk(TG/B)$ restricted to the fixed points can be considered an element of
\[
\mathcal M=
\bigoplus_{\sigma\in W} \CC(\T\times \T^*\times \C^*)[[q]],
\]
where $\CC(\T\times \T^*\times \C^*)$ stands for the field of rational functions.
We consider three actions of the Weyl group on $\mathcal M$:
\begin{itemize}
\item $W$ acts on $\mathcal M$ by permuting the components. For a reflection $s\in W$ the action is 
$$\left\{f_\sigma\right\}_{\sigma\in W}\quad\mapsto\qquad \left\{f_{\sigma s}\right\}_{\sigma\in W}.$$
This action on $\mathcal M$ will be denoted by $s^\gamma$.


\item $W$ acts on $\T$ and hence on $\CC(\T)\simeq \C(z_1,z_2,\dots,z_n)$ (the $z_i$ variables will be called the ``equivariant variables''). This action will be denoted by $s^z$.

\item $W$ acts on the space of characters and on the quotient torus $\T^*$. The resulting action on $\CC(\T^*)$ will be denoted by $s^\mu$.
\end{itemize}

Our first result is a recursive formula for the elliptic class.
It takes the  most elegant form when expressed 
 in the equivariant elliptic cohomology of $G/B$. Let $\alpha\in \t^*$ be a simple root, then

\begin{formula}\label{BS-intro}\begin{multline*}
\delta\left(\LL{\alpha},h^{\alpha^\vme}\right) 
\cdot s^\mu_\alpha \Ee(X_\om) -
\delta\left(\LL{\alpha},h\right)\cdot
s^\gamma_\alpha s^\mu_\alpha \Ee(X_\omega)=\\ 
=\begin{cases}\Ee(X_{\om s_\alpha})& \text{\rm if } \ell(\om s_\alpha)=\ell(\om)+1,\\ \\
\delta(h^{\alpha^\vme},h)\delta(h^{-\alpha^\vme},h)\cdot\Ee(X_\omega)
&\text{\rm if }\ell(\om s_\alpha	)=\ell(\om)-1.\end{cases}\end{multline*}\end{formula}
\noindent Here
\begin{itemize}
\item $\alpha^\vme$ is the dual root, the expression $h^{\alpha^\vme}$ is a function on $\t^*$,
\item $s_\alpha\in W$ is the  reflection in $\alpha$,
\item $\LL{\alpha}=G\times_B\C_{-\alpha}$ is the line bundle associated to the root $\alpha$,
\item $\delta(x,y)=\frac{\vt'(1)\vt{(x y)}}{\vt(x)\vt(y)}$  a certain function defied by the multiplicative version of the Jacobi theta function. 
\end{itemize}
The proof relies on the study of the Bott-Samelson inductive resolution of the Schubert variety. 

The recursion above can be studied in the framework of Hecke algebras. In Section \ref{sec:Hecke}  we present a version of a Hecke algebra which acts on the elliptic classes. We will also take this opportunity to explore the various degenerations of the elliptic class (and corresponding Hecke operations), such as Chern-Schwartz-MacPherson class, motivic Chern class, and cohomological and K theoretic fundamental class.


We prove that, in addition to the Bott-Samelson recursion above, elliptic classes of Schubert varieties satisfy another ``dual'' recursion in equivariant elliptic cohomology of $G/B$:
\begin{formula}\label{R-intro}\begin{multline*}
\delta\left(\eee^{-\alpha},h^{\om^{-1}\alpha^\vme}\right) 
\cdot \Ee(X_\om) - 
\delta\left(\eee^\alpha,h\right)\cdot s_\alpha^z
\Ee(X_\omega)=\\
=\begin{cases}\Ee(X_{s_\alpha\om})& \text{\rm if } \ell(s_\alpha\om)=\ell(\om)+1,\\ \\
\delta(h^{\omega^{-1}\alpha^\vme},h)\delta(h^{-\omega^{-1}\alpha^\vme},h)\cdot
\Ee(X_{s_\alpha\om})&\text{\rm if }\ell(s_\alpha\om)=\ell(\om)-1.\end{cases}
\end{multline*}\end{formula}

Remarkably, if $G=\GL_n$ then this recursion is equivalent to a three term identity in \cite{RTV} (and references therein), where this three term identity is interpreted as the R-matrix identity for an elliptic  quantum group. Hence we will call this second dual recursion the {\em R-matrix recursion}. 

\medskip

Let us introduce the rescaled elliptic classes  
\[
\Em(X_\om)=
\prod_{\alpha\in \Sigma_-\cap\, \omega^{-1}\Sigma_+} \delta(h^{-\alpha^\vme},h)^{-1}\cdot \Ee(X_\om).
\]
where $\Sigma_\pm$ is the set of positive/negative roots. In terms of this version both of our recursions (Formulas  \ref{BS-intro} and \ref{R-intro}) can be expressed in more compact forms:
\begin{theorem}[Bott-Samelson recursion] 
Let $\alpha$ be a simple root. Then
\[
\Em(X_{\om s_\alpha}) =
\frac{\delta( \LL{\alpha},h^{\alpha^\vme})}
{\delta(h^{-\alpha^\vme},h)}\cdot 
s_\alpha^\mu\Em(X_{\om})  
-
\frac{\delta(\LL{\alpha},h)}{\delta(h^{-\alpha^\vme},h)}\cdot
s_\alpha^\gamma s_\alpha^\mu\Em(X_{\om}).
\]
If $G=\GL_n$ then in the language of weight functions
$$
\LL{\alpha_k}=\tfrac{\gamma_{k+1}}{\gamma_k},\qquad h^{\alpha_k^\vme}= \tfrac{\mu_{k+1}}{\mu_k}.
$$
\end{theorem}

\begin{theorem}[R-matrix recursion]  
Let $\alpha$ be a simple root. Then
\[
\Em(X_{s_\alpha \om}) =
\frac{\delta(\eee^{-\alpha}, h^{\omega^{-1}\alpha^\vme})}
{\delta(h^{-\omega^{-1}\alpha^\vme},h)} \cdot
\Em(X_{\om}) 
+
 \frac{\delta(\eee^\alpha,h)}
{\delta(h^{-\omega^{-1}\alpha^\vme},h)} \cdot s_\alpha^z\Em(X_\om)
,\]
where $\eee^\alpha\in K_\T(pt)\simeq {\rm R}(\T)$ is the character corresponding to the root $\alpha$. If $G=\GL_n$ then $\eee^{\alpha_k}=\frac{z_k}{z_{k+1}}$.
\end{theorem}

\subsection{Relation to weight functions}

In \cite{RTV} certain special functions called {\em weight functions} are defined that satisfy the R-matrix recursion. These weight functions are the elliptic analogues of cohomological and K theoretic weight functions whose origin goes back to \cite{TV}. The three versions of weight functions play an important role in representation theory, in the theory of hypergeometric solutions of various KZ differential equations, and also turn up in Okounkov's theory of stable envelopes. 

The fact that elliptic classes of Schubert varieties satisfy the R-matrix recursion allows us to  prove
\begin{theorem}
The elliptic classes $\Ee(X_\om)$ for $G=\GL_n$ are represented by weight functions $\wwh_\om$, that is, 
$$\Ee(X_\om)=\eta(\wwh_\om),$$
where $\eta:K_{\T\times\T}(\End(\C^n))\to K_{\T\times\T}(\GL_n)\simeq K_\T(\GL_n/B)$ is the restriction map, which is  a surjection, composed with Chern character. 
\end{theorem}


\subsection{By-products} 
Throughout the paper we heavily use equivariant localization, that is, we work with torus fixed point restrictions of the elliptic classes scaled by an Euler class:
$$E_\sigma(X_\om)=\frac{\Ek(X_\om,\lambda)_{|\sigma}}{\eekk(T_\sigma(G/B))},$$
where $\sigma$ is a torus fixed point. Thus our proofs are achieved by calculations of restricted classes that live in the equivariant K theory of fixed points. 
We should emphasize that many of our formulas encode deep identities among theta functions (occasionally we will remark on Fay's three term identity, and a four term identity). For us, these identities will come for free, we will not need to prove them. This is the power of Borisov and Libgober's theory: their classes are well defined. 

\medskip

A remarkable by-product of the fact that the $E(X_\om)$ classes satisfy two seemingly unrelated recursions is the fact that weight functions, besides satisfying the known R-matrix recursions, also satisfy a so far unknown recursion coming from the Bott-Samelson induction. This will be presented in Section \ref{sec:tale}, and will be interpreted as the R-matrix relation for the 3d mirror dual variety in a follow up paper (for a special case see \cite{Sm2}). 

\subsection{ Conventions.} We work with varieties over $\C$; for example by $\GL_n$ we mean $\GL_n(\C)$. By K-theory we mean algebraic K theory of coherent sheaves, which in our case is isomorphic to the topological equivariant K theory, see the Appendix of \cite{FRW}. The Weyl group of $\GL_n$ will be identified with the group of permutations. A permutation $\om:\{1,\ldots,n\}\to \{1,\ldots,n\}$ will be denoted by a list $\om=\om(1)\ldots\om(n)$. The permutation switching $k$ and $k+1$ will be denoted by $s_k$. Permutations will be multiplied by the convention that $\om \sigma$ is the permutation obtained by first applying $\sigma$ then $\om$; for example $231 \cdot 213=321$.
\medskip

For a general simply connected reductive group $G$ we fix a maximal torus $\T$ and a Borel group $B$ containing $\T$. The set of roots of $B$, according to our convention, are positive. For an integral weight $\lambda\in\t^*$ the line bundle \begin{equation}\label{def:LL}\LL{\lambda}=G\times_B\C_{-\lambda}\end{equation} is such that the torus $\T$ acts via the character $-\lambda$ on the fiber at $eB\in G/B$. With this convention $\LL{\lambda}$ is ample for $\lambda$ belonging to the interior of the dominant Weyl chamber. The half sum of positive roots is denoted by $\rho$. The canonical divisor  $K_{G/B}$ is isomorphic to $\LL{\rho}^{-2}$.  
\medskip

\noindent{\bf Acknowledgment.} We are grateful to Jaros{\l}aw Wi{\'s}niewski for very useful conversations on Bott-Samelson resolutions, and to Shrawan Kumar for discussions on the canonical divisors of Schubert varieties and their relation with Bott-Samelson varieties.

\section{The equivariant elliptic characteristic class twisted by a line bundle}\label{sec2}


First we recall the notion of the elliptic characteristic class of a singular variety, defined by 
Borisov and Libgober in \cite{BoLi1}. We study its equivariant version, and its behavior with respect to fixed point restrictions. In Section \ref{sec:twisted_gen} we define a version ``twisted by a line bundle and its section.''  This latter version will be used in the rest of the paper.

The original definition of elliptic genus for singular varieties arose from the study of mirror symmetry for a generic hypersurface in the toric variety associated to a reflexive polytope, \cite{BoLi0}. For a possibly singular variety $X$ with a Weil divisor $\Delta$ (which satisfy some assumptions, see the definition below) the elliptic genus was constructed by Borisov and Libgober in \cite{BoLi1}. Their construction goes as follows. Let  $f:Z\to X$ be a log-resolution of the pair $(X,\Delta)$ and $D=f^*(K_X+\Delta)-K_Z$. The genus is computed as an integral of a certain class $\Ellt(Z,D)$ defined in terms of the Chern classes of $Z$ and the divisor $D$. An alchemy of the formulas make the definition independent of the resolution. The key argument is that whenever we have a blowup in a smooth center $b:Z_1\to Z_2$ then $b_*(\Ellt(Z_1,D_1))=\Ellt(Z_2,D_2)$. By weak factorization theorem (i.e.~since any two resolutions differ by a sequence of blow-ups and blow-downs in smooth centers) the push-forward to the point of $\Ellt(Z,D)$ does not depend on the resolution. In fact more is obtained: the push-forward to $X$ does not depend on the resolution. This way homology classes are defined, since $X$ is singular in general. If $X$ is embedded in a smooth ambient space, then it is more convenient to consider (by Poincar\'e duality) the dual class in the cohomology of the ambient space. 

Below we give details of the sketched construction, introducing some simplifications and moving all the objects to K theory, from where they in fact come via the Chern character.

\subsection{Theta functions} \label{sec:theta}
For general reference on theta functions see e.g. \cite{We,Cha}. We will use the following version
\[
\vt(x)=\vt(x,q)=
x^{1/2}(1-x^{-1})\prod_{n\geq 1}(1-q^n x)(1-q^n /x)\in \Z[x^{\pm 1/2}][[q]].
\]
For a fixed $|q|<1$ the series is convergent and defines a holomorphic function on a double cover on $\CC^*$.
Throughout the paper we will use the function
\begin{equation}\label{def:delta}
\delta(a,b)=\frac{\vt(ab)\vt'(1)}{\vt(a)\vt(b)},
\end{equation}
which is meromorphic on $\C^*\times \C^*$ with poles whenever $a$ or $b$ is a power of $q$. 
As a power series in $q$ the coefficients of $\delta(a,b)$ are rational functions in $a$ and $b$
$$\delta(a,b)=\frac{1-a^{-1}b^{-1}}{(1-a^{-1})(1-b^{-1})}+q(a^{-1}b^{-1}- ab)+q^2(a^{-2}b^{-1}+a^{-1}b^{-2}- a^2b- ab^2)+\dots\,.$$
We will also use theta functions in additive variables, namely, let
$$
\theta(u)=\theta(u,\tau)=\vt(\eee^{u},q) 
$$
where $q=\eee^{2\pi i\tau}$,  $\tau\in \CC$, $im(\tau)>0$. 
Our function $\theta$ differs from the classical Jacobi theta-function only by a factor depending on $\tau$. Namely, according to Jacobi's product formula \cite[Ch~V.6]{Cha},
\begin{equation*}\label{Jacobi-prod-for}
\theta_\text{Jacobi}(u,\tau)=2\eee^{\pi i\tau/4}
 \sin (\pi u)\prod_{\ell=1}^{\infty}(1-\eee^{2\pi \ell i\tau})(1-\eee^{2\pi\ell i(\tau+u)}  )(1-\eee^{2\pi\ell i(\tau-u)} ),
\end{equation*}
and hence
\[
\theta_\text{Jacobi}(u,\tau)=\frac1iq^{1/8}\prod_{\ell=1}^{\infty}(1-q^{\ell})\cdot
\theta(2\pi i u,\tau).
\]
The theta function satisfies the quasi-periodicity identities
\begin{align}\label{theta_trans}
\theta(u+2\pi i,\tau)&=-\theta(u,\tau),\\
\notag \theta(u+2\pi i\tau,\tau)&=- \eee^{- u-\pi i \tau}  \cdot \theta(u,\tau).  
\end{align}
Note that the periods of our theta function differ by the factor $2\pi i$ comparing with the Jacobi theta function. Our convention fits well to the topological context, where $\vt$ is composed with the Chern character. The same convention was chosen for example in \cite{Hirzebruch, Hohn, HirzebruchBook}. 
We will take a closer look at transformation properties of several variable functions built out of theta functions in Section \ref{sec:transformations}.

\subsection{The elliptic class of a smooth variety}

First we define the elliptic class of a smooth variety $Z$, cf. \cite[Lemma 2.5.2]{Hohn}, \cite[\S3]{Totaro}, \cite[Formula (3)]{BoLi0}. 
For a rank $n$ bundle $\TTT$ over $Z$ with \gr roots $t_i$ (that is, $\sum_{i=1}^nt_i=[\TTT]$ in K theory) define
\begin{align}
\notag \Lambda_x(\TTT) & =\sum_{k=0}^{\rank \TTT}[\Lambda^k\TTT] x^k=\prod_{i=1}^n (1+xt_i) \   \in K(Z)[x],  \\
\notag S_x(\TTT) & =\sum_{k=0}^{\infty}[S^k\TTT] x^k=\prod_{i=1}^n \frac{1}{(1-xt_i)} \ \in K(Z)[[x]], \\
\label{eq:eK} \eekk(\TTT) & =\Lambda_{-1}(\TTT^*)=\prod_{i=1}^n (1-t_i^{-1}) \ \ \ \ \ \ \in K(Z).
\end{align} 
The last one plays the role of Euler class in K theory.

\begin{definition}\label{def:Ell0}
For a smooth variety $Z$ with tangent bundle $TZ$ whose \gr roots are $t_i$, define its {\em elliptic bundle} 
\begin{align*}
\ELL(Z)  = & 
y^{-\dim Z/2}\, \Lambda_{-y}\Omega^1_Z
\otimes\bigotimes_{n \ge 1}
\Bigl(\Lambda_{-yq^{n}}(\Omega^1_Z) \otimes \Lambda_{-y^{-1}q^n}
(TZ) \otimes S_{q^n}(\Omega^1_Z) \otimes
S_{q^n}(TZ) \Bigr) \\
 = &
y^{-\dim Z/2}\prod_{i=1}^{\dim Z}(1-yt_i^{-1})\prod_{n\geq 1}\frac{(1-yq^{n}t_i^{-1})(1-y^{-1}q^{n}t_i)}{(1-q^nt_i^{-1})(1-q^nt_i)} \\
 = &
\eekk(TZ) \prod_{i=1}^{\dim Z}\frac{\vt(t_i /y)}{\vt(t_i)}  
\hskip 7 true cm \in K(Z)[y^{\pm1/2}][[q]].
\end{align*}
\end{definition}

As we indicted in the notation $K(Z)[y^{\pm1/2}][[q]]$, the class $\ELL(Z)$ is a formal series in $q$, or equivalently, in $\eee^{2\pi i\tau}$. 
Since the power series defining the theta function converges for $|q|<1$ the class $\ELL(Z)$ can be considered as a function on $(y,q)\in \C^*\times \C_{\rm im >0}$ with values in a suitable completion of K theory. We regard $q$ or $y$ as a parameter of the theory, and hence  we do not  indicate this dependence, unless we want to emphasize it. 
\medskip

\subsection{The Borisov-Libgober class $\ELLt$} \label{sec:shelf}
For a singular pair $(X,\Delta)$ Borisov and Libgober defined an elliptic class $\Ellt(X,\Delta)$, living in cohomology. In this section we introduce a K theoretic modification $\ELLt(X,\Delta)$ of their construction. The relation to the original definition in \cite{BoLi1} is $\Ellt(X,\Delta)=td(TM) ch(\ELLt(X,\Delta))$.

Consider pairs $(Z,D)$  where $Z$ is smooth and projective and $D$ is
a SNC $\QQ$-divisor on $Z$ i.e. $D=\sum_{k=1}^\ell a_kD_k$ is a formal sum
such that the components $D_k$ are smooth
divisors on $Z$ intersecting transversely, $a_k\in\QQ$. We additionally require $a_k \neq 1$ for all $1\leq k\leq \ell$. Equally well we can consider divisors with complex coefficients.

\begin{definition}
Define
\begin{equation*}\label{ellgenpairsformula}
\ELLt(Z,D)=
\eekk(TZ)\prod_{i=1}^{\dim Z} \frac{\vt(t_ih) \vt'(1)} {\vt(t_i)\vt (h)} \prod_{k=1}^{\ell} \frac{\vt(d_k h^{1-a_k}) \vt(h)}{\vt(d_k h) \vt(h^{1-a_k})}\in K_\T(X)(h^{ 1/N})[[q]],
\end{equation*}
where $d_k=[{\mathcal O}(D_k)]$ and $t_i$ are the Grothendieck roots of $TZ$. Here $N$ is an integer divisible by all the denominators of the rational numbers $a_k$. 
\end{definition}

\begin{remark}[About notation]\rm Borisov and Libgober use the notation
$D=-\sum_{k=1}^\ell \alpha_kD_k$,  
but we reserve the letter $\alpha$ for a root of a Lie algebra, and we want to get rid of the minus sign. 
Another change is that we introduce the letter $h=y^{-1}$. This way we avoid the following inconsistency:  when $q\to 0$ then the elliptic genus converges to Hirzebruch genus $\chi_{-y}$, not to $\chi_y$. 
In H\"ohn's definition 
$y$ appears with a different sign. Our main reason of passing from $y$ to $h$ is to  agree with the notation of \cite{RTV} and with the circle of papers related to Okounkov theory. 
\end{remark}

\begin{remark} \rm 
Formally there are rational or even possibly complex powers of $h$ in the expression above. By this, here and in the whole paper we mean the following: for a formal variable $z$ satisfying $h=\eee^{-z}$, by  $\vt(h^a)$ we mean  
\[
\vt(h^a)=\vt((\eee^{- z})^a)=\vt(\eee^{- za})=\theta(-za). 
\]
To keep the exposition simple we will not explicitly work with the $z$ variable anymore, and this kind of dependence on formal powers of $h$ we indicate by $K_{\T}(X)(h^{ 1/N})$. 
\end{remark}

Observe that if the divisor $D$ is empty then
\[
\ELLt(Z,\emptyset)=\left(\tfrac{\vt'(1)}{\vt(h)}\right)^{\dim Z}\ELL(Z).
\]

The two definitions above generalize without change to the torus equivariant case: the case when $\T=(\CC^*)^n$ acts on $Z$ or $(Z,D)$, and $t_i$ are the $\T$ equivariant \gr roots. One advantage of torus equivariant K theory is the tool of fixed point restrictions. Let $x$ be a $\T$ fixed point on $Z$, and consider the restriction of $\ELLt(Z,D)$ to $x$. Here each $d_k$ can be chosen as a \gr root of the tangent bundle. By introducing artificial components of $D$ with coefficient $a_i=0$ if necessary, we may now assume that $d_i=t_i$ for all $i=1,\ldots,\dim Z$, and we obtain 
\begin{align}\label{fixed} \notag
\ELLt(Z,D)_{|x}= & \eekk(TZ)\prod_{i=1}^{\dim Z} \frac{\vt(t_i h^{1-a_i})\vt'(1) }  {\vt(t_i)\vt(h^{1-a_i})} \\
 = & \eekk(TZ)\prod_{i=1}^{\dim Z} \delta(t_i, h^{1-a_i}) \hskip 1 true cm \in K_{\T}(x)(h^{1/N})[[q]]=\RR(\T)(h^{1/N})[[q]].
\end{align}

Now we are ready to define the elliptic class for certain pairs of singular varieties.  Let $X$ be a possibly singular subvariety of a smooth variety $M$, and $\Delta$ a divisor on $X$. We say that $(X,\Delta)$ is a KLT pair, if 
\begin{itemize} 
\item $K_X + \Delta$ is a $\QQ$-Cartier divisor;
\item there exists a map $\f:Z\to M$ which is a log-resolution $(Z,D)\to(X,\Delta)$ (i.e. $Z$ is smooth, $D$ is a normal crossing divisor on $Z$, $\f$ is proper and is an isomorphism away from $D$) such that 
\begin{enumerate}[(i)]
\item  $D=\sum_ka_kD_k$, $a_k<1$,\footnote{In the classical definition it is required that $\Delta=-\sum \alpha_i \Delta_i$ with $\alpha_i\in(-1,0]$ (hence $a_i\in [0,1)$), but we do not need positivity of $a_i$'s.}
\item $K_Z+D = \f^*(K_X + \Delta)$.
\end{enumerate}
\end{itemize}

\begin{definition}\label{def:KLTEll}
The elliptic class of a KLT pair $(X,\Delta)$ is defined by
$$\ELLt(X,\Delta;M):=\f_*(\ELLt(Z,D)) \in K(M)(h^{1/N})[[q]].$$
\end{definition}



The key result in \cite{BoLi1} goes through without major changes to show that $\ELLt(X,\Delta;M)$ as  defined here is independent of the choice of the resolution. When allowing complex coefficients $a_k$, the proof of independence is unchanged as long as $a_k\not\in \NN_{\geq 1}$. Note however that to make sense of the definition we have to ensure that some $\C$-multiple of $K_X+\Delta$ is a Cartier divisor. In the presence of a torus action $\ELLt(X,\Delta;M)$ is defined formally the same way, see details in \cite{Wae, DBW}. We will not indicate the torus action in the notation.

\begin{remark}\rm To avoid the dependence of an embedding into an ambient space $M$ and work directly with the singular space $X$ we would have to deal with the K theory of coherent sheaves denoted in the literature by $G(X)$. This complication is unnecessary for our purposes.\end{remark} 

\begin{remark} \rm 
The elliptic class is the ``class version'' of the {\em elliptic genus} studied in detail in the literature. Namely, let $h=\eee^{-z}$, 
and let $Z$ be Calabi-Yau. Then the elliptic genus $\eta_*(\ELL(Z))$ ($\eta:Z\to$ point) is a holomorphic function on $\CC_z\times \CC_{\rm im \tau>0}$, and it is a quasi-modular form of weight 0 and index $\frac{\dim Z}2$. 
Now let us assume that $N$ times the multiplicities of $\Delta$ are integers, and that $(X,\Delta)$ is a Calabi-Yau pair (i.e. $K_X+\Delta=0$). Then the ``genus''
\[
\left(\tfrac{\theta(-z,\tau)}{\theta'(0,\tau)}\right)^{\dim X} \eta_*( \ELLt(X,\Delta;M))
\]
has transformation properties of Jacobi forms of weight $\dim X$ and index $0$ \cite[Prop. 3.10]{BoLi1}, with respect to the subgroup of the full Jacobi group 
 generated by the transformations\footnote{Since our theta function is quasiperiodic with respect to the lattice $2\pi i\langle1,\tau\rangle$, we have rescaled the formulas with respect to the those appearing in \cite{BoLi1}.}
$$\begin{array}{lll}(z, \tau ) \mapsto  (z +2\pi i N, \tau ),&& (z, \tau ) \mapsto (z + 2N\pi i \tau, \tau ), \\ \\
 (z, \tau ) \mapsto (z, \tau + 1), && (z, \tau ) \mapsto (z/(2\pi i\tau), -1/\tau ).\end{array}$$
\end{remark}

\subsection{Calculation of the elliptic class via torus fixed points} \label{sec:localization}

The class $\ELLt(X,\Delta;M)$ is defined via the push-forward map $\f_*$. From the well known localization formulas (which we call Lefschetz-Riemann-Roch theorem) for push-forward in torus equivariant K theory, see e.g.~\cite[Th. 5.11.7]{ChGi}, we obtain the following proposition.

\begin{proposition} \label{prop:localization_general}
 Assume that in the $\T$ equivariant situation of Definition \ref{def:KLTEll} the fix point sets $M^{\T}$ and $Z^\T$ are finite. Then for $x\in M^\T$ in the fraction field of $K_\T(x)(h^{1/N})[[q]]$ we have
\begin{equation}\label{EllResLoc}
\frac{\ELLt(X,\Delta;M)_{|x}}{\eekk(T_xM)}= \sum_{x'\in \f^{-1}(x)} \frac{\ELLt(Z,D)_{|x'}}{\eekk(T_{x'}Z)}.
\end{equation} 
\qed
\end{proposition}

This formula motivates the definition of the local elliptic class
\begin{equation}\label{eq:Elocal}
\EE_x(X,\Delta;M)=\frac{\ELLt(X,\Delta;M)_{|x}}{\eekk(T_xM)}
\end{equation}
for a fixed point $x$ on $X$. In fact the division by $\eekk(T_xM)$ gets rid of the dependence on $M$, so we set $\EE_x(X,\Delta)=\EE_x(X,\Delta;M)$ for some $M$. Indeed, having two equivariant embeddings $\iota_i:X\hookrightarrow M_i$ for $i=1,2$ we obtain the third, the diagonal one $\iota:X\hookrightarrow M=M_1\times M_2$. Let  $x\in X^\T$ be an isolated fixed point. Then by Lefschetz-Riemann-Roch theorem for the projection $\pi_i:M\to M_i$ we have
$$\frac{\ELLt(\iota(X),\iota(\Delta);M)}{\eekk(TM)}=\frac{\ELLt(\iota_i(X),\iota_i(\Delta);M_i)}{\eekk(TM_i)}$$
for $i=1,2$.

From Proposition \ref{prop:localization_general} we obtain 
\begin{equation}\label{EllResLoc1}
\EE_x(X,\Delta)  = \sum_{x'\in \f^{-1}(x)} \EE_{x'}(Z,D)  = \sum_{x'\in \f^{-1}(x)}\prod_{k=1}^{\dim Z}\delta(t_k(x'),h^{1-a_k(x')}), 
\end{equation} 
where $t_k(x')$ and $a_k(x')$ denote the tangent characters and multiplicities of the divisors at the torus fixed point $x'$.

\begin{remark} \rm \label{rem:EllH} 
Our choice in this paper is to use equivariant K theory as the home of the elliptic characteristic classes
and we have two possible ways of expressing the elliptic genus:
$$p_*^{\rm K} \ELLt(X,\Delta;M)=\int_M td(M)ch(\ELLt(X,\Delta;M)),$$
where $p:M\to pt$ and $p_*^K$ the is  push-forward in the K theory, i.e.~$\chi(M,-)$.
 We could have decided differently by setting up equivariant {\em elliptic cohomology} 
to be the rational Borel equivariant cohomology, extended by the formal variable $q$ and with the Euler class given by the theta function. In that contexts we would define the elliptic class $\Ellt^{\rm ell}(X,\Delta;M)$ as the push forward (in elliptic cohomology) of the suitable class defined for a resolution. The resulting class satisfies
\begin{equation}\label{formaldef}
\Ellt^{\rm ell}(X,\Delta;M)=\frac{\eell(TM)}{\eehh(TM)}\Ellt(X,\Delta;M)
\end{equation}
according to the general Grothendieck-Riemann-Roch theorem. Not going into details let us take~\eqref{formaldef} as the definition of the {\em elliptic class in equivariant elliptic cohomology}.
We  have
$$p_*^{\rm ell} \Ellt^{\rm ell}(X,\Delta;M)=\int_M \frac{\eehh(TM)}{\eell(TM)}\Ellt^{\rm ell}(X,\Delta;M)=\int_M \Ellt(X,\Delta;M)\,.$$
Observe that the quotient in \eqref{eq:Elocal}
$$\EE_x(X,\Delta)=\frac{\Ellt^{?}(X,\Delta;M)}{e^?(TM)}$$
is not only independent of $M$ but it
is also essentially independent  of the cohomology theory used. The only thing that has to by changed when passing from K theory to cohomology or Borel elliptic theory are the substitutions   $h=\eee^{-z}$ 
and $z_i=\eee^{x_i}$, where $z_i$'s are the basic characters $\T\to \C^*$ and $x_i\in \t^*$ the basic weights.
The use of K theory is more economic, since there the classes $\EE_x(X,\Delta)$ are power series in $q$ with coefficients in rational functions in $z_i$ and roots of $h$, depending on the denominators of the multiplicities $a_i$:
$$\EE_x(X,\Delta)=Frac\big(\RR(\T)(h^{1/N})\big)[[q]]\,.$$
\end{remark}

\begin{example} \rm
Consider the standard action of $\T=(\CC^*)^2$ on $M=X=\C^2$. Denote $\C_x=\{x=0\}$, $\C_y=\{y=0\}$, and let them represent the classes $t_1$ and $t_2$ in $K_\T(\C^2)$. Consider the divisor $\Delta=a_1\C_x+a_2\C_y$.
Taking the identity map as resolution, using  \eqref{fixed} we obtain 
\[
\EE_0(X,\Delta)=
\delta(t_1,h^{1-a_1})  \delta(t_2,h^{1-a_2}).
\]
Now we calculate $\ELLt(X,\Delta;M)$ in another way. Consider the blow-up $Z=Bl_0X$, $\f:Z\to X$ with exceptional divisor $E$, and let the strict transforms of $\C_x, \C_y$ be $\tilde{\C}_x,\tilde{\C}_y$. Define the divisor $D=a_1\tilde{\C}_x +a_2\tilde{\C}_y+(a_1+a_2-1)E$. Since $\f^*K_X=K_Z-E$, and $\f^*(\Delta)=a_1\tilde{\C}_x + a_2\tilde{\C}_y+(a_1+a_2)E$ (from calculations in coordinates) we have $\f^*(K_X+\Delta)=K_Z+D$.
Hence, by \eqref{EllResLoc} and using \eqref{fixed} at the two $\T$ fixed points in $\f^{-1}(0)$ we obtain
\[
\EE_0(X,\Delta) =  \delta(t_1,h^{2-a_1-a_2}) \delta(t_2/t_1,h^{1-a_2}) + \delta(t_2,h^{2-a_1-a_2}) \delta(t_1/t_2,h^{1-a_1}).
\]
Observe that the comparison of the two formulas above 
boils down to the identity 
\[
\delta(t_1,h^{1-a_1})  \delta(t_2,h^{1-a_2})=
 \delta(t_1,h^{2-a_1-a_2}) \delta(t_2/t_1,h^{1-a_2}) + \delta(t_2,h^{2-a_1-a_2}) \delta(t_1/t_2,h^{1-a_1}),
\]
which is a rewriting of the well known {\em Fay's trisecant identity}
\begin{multline}
 \theta(a+c)\theta(a-c)\theta(b+d)\theta(b-d)=\\
 \theta(a+b)\theta(a-b)\theta(c+d)\theta(c-d)+ \theta(a+d)\theta(a-d)\theta(b+c)\theta(b-c),
\label{trisecant}\end{multline}
see e.g. \cite{Fay}, \cite[Thm. 5.3]{FRV}, \cite{GL}.
\end{example}

\subsection{The elliptic class $\ELLt$ twisted by a line bundle}\label{sec:twisted_gen}

For the main application of the present paper (elliptic classes of Schubert varieties) we need a modified version of the notion $\ELLt(X,\Delta;M)$. The following example explains why.

\begin{example} \rm
Let $X_{\om}\subset G/B$ be a Schubert variety, and let $f_\omm:Z_\omm\to X_{\om}$ be a Bott-Samelson resolution (as defined in Section \ref{BoSa}).
Set 
\[
\Delta=\partial X_{\om}=\sum_{X_{\om'}\subset X_{\om},\;\dim(X_{\om'})=\dim(X_{\om})-1} [X_{\om}].
\] 
It is a fact that $K_{X_{\om}}+\Delta$ is a Cartier divisor and $f_\omm^*(K_{X_{\om}}+\Delta)=K_{Z_\omm}+D,$ where $D=\partial Z_\omm$ has all coefficients  equal to 1, see Theorem \ref{fullflagcnonical} below. 
Hence a crucial condition in the definition of $\ELLt$ is not satisfied for $(X_\om,\partial X_\om)$.
\end{example}

Let $X$ be a possibly singular subvariety of a smooth variety $M$, and $\Delta$ a divisor on $X$. Assume that 
$K_X + \Delta$ is a $\QQ$-Cartier divisor; and that there exists a map $\f:Z\to M$ which is a log-resolution $(Z,D)\to(X,\Delta)$ such that $K_Z+D = \f^*(K_X + \Delta)$. Observe that we have not assumed anything about the coefficients of the divisor $D$ (cf. the definition of KLT pair in Section \ref{sec:shelf}).

Let $L$ be a  line bundle on $X$ and $\xi\in H^0(X;L)$ a section such that $\xi$ does not vanish on $X\setminus supp(\Delta)$. Assuming $supp(\Delta)=supp(zeros(\xi))$ we have $supp(f^*D)=supp(zeros(f^*\xi))$. Denote 
\[
\Delta(L,\xi)=\Delta-zeros(\xi).
\]

If $L$ is sufficiently positive then the pair $(X,\Delta(L,\xi))$ is a KLT pair. Therefore we can define the twisted elliptic class of $(X,\Delta)$. The definition makes sense for a fractional power of $L$. Then $\ELLt(X,\Delta(L,\xi);M) \in K(M)(h^{1/N})[[q]]$ for a sufficiently divisible $N$.

\begin{remark} \rm
In practice, see Section \ref{sec:twisted} below, the line bundle $L$ will be associated with certain integer points $\lambda$ of a vector space, 
$V$ consisting of Cartier divisors supported by $supp(\Delta)$. The dependence of the twisted elliptic class on $\lambda\in V$ will be meromorphic; so we can extend the definition to a meromorphic function on $V$: the twisted elliptic class of $(X,\Delta)$ will be a class defined for almost all $\lambda$, hence understood as a meromorphic function on $V$.
\end{remark}


\begin{remark}\rm 
It is tempting to think that the ``right'' elliptic class notion for a Schubert variety $X_\om$ is either $\ELLt(X_\om,\emptyset)$ or $\ELLt(X_\om,\partial X_\om)$, and the trick of ``twisting with a line bundle'' in this section is not necessary. However, in general we cannot take $\Delta=\emptyset$, since $X$ may not have $\QQ$ Gorenstein singularities, and then $f^*(K_X)$ does not make sense.  This is the case for most of the Schubert varieties. The pair $(X_\om,\partial X_\om)$ is Gorenstein, i.e.~the divisor $K_{X_\om}+\partial X_\om$ is Cartier (see Theorem \ref{fullflagcnonical} below), but the multiplicities $a_i$ take the forbidden value 1. 
\end{remark}

\section{Bott-Samelson resolution and the elliptic classes of Schubert varieties}
\label{BoSa}

In this section we define elliptic classes of Schubert varieties following the general line of arguments in Section \ref{sec2}. For this, after introducing the usual settings of Schubert calculus, we describe a resolution of Schubert varieties inductively. 

\smallskip

Let $G$ be semisimple group with Borel subgroup $B$, maximal torus $\T$, and Weyl group $W$. Simple roots will be denoted by $\alpha_1,\alpha_2,\ldots,$ and the corresponding reflections in $W$ by $s_1,s_2, \ldots$.  We consider reduced words in the letters $s_k$, denoted by $\omm$. A word $\omm$ represents an element $\om\in W$. The length $\ell(\om)$ of $\om$ is the length of the shortest reduced word representing it.

We will study the homogeneous space $G/B$. For $\om\in W$ let $\tilde{\om}\in N(\T)\subset G$ be a representative of $\om\in W=N(\T)/\T$, and let $x_\om=\tilde{\om}B \in G/B$. The point $x_\om$ is fixed under the $\T$ action. The $B$ orbit $X^\circ_\om=Bx_\om=B\tilde{\om}B$ of $x_\om$ will be called a Schubert cell, and its closure $X_\om$ the Schubert variety. In this choice of conventions we have $\dim(X_{\om})=\ell(\om)$.

\subsection{The Bott-Samelson resolution of Schubert varieties}\label{sec:BSres}
Let the reduced word $\omm$ represent $\om\in W$. The Bott-Samelson variety $Z_{\omm}$, together with a resolution map $f_{\omm}: Z_{\omm}\to X_{\om}$ of the Schubert variety $X_{\om}$ is constructed inductively as follows. Suppose $\omm=\omm's_k$ is a reduced word. 
Let $P_k\supset B$ be the minimal parabolic containing $B$, such that $W_{P_k}=\langle s_k\rangle$. The map $\pi_k:G/B\to G/P_k$ is a $\PP^1$ fibration. It maps the open cell $X^\circ_{\om'}$ isomorphically to its image. 
We have $X_{\om}=\pi_k^{-1}\pi_k(X_{\om'})$ and $\pi_k$ restricted to  $X^\circ_\om$ is an $\AA^1$ fibration. 
The variety $Z_\omm$ fibers over $Z_{\omm'}$ with the fiber $\PP^1$. 
We have a pull-back diagram
\begin{equation}\label{BS-diagram}\xymatrixcolsep{3pc}\xymatrix{Z_\omm\ar[d]^{\pi_{\omm}}\ar[r]^{f_\omm}&X_{\om} \ar[d]\ar@{^{(}->}[r]& G/B \ar[d]^{\pi_k}\\
Z_{\omm'}\ar@/^1pc/[u]^\iota\ar[r]^{\pi_k\circ f_{\omm'}\phantom{aa}}&\pi_k(X_{\om'})\ar@{^{(}->}[r]& G/P_k. \\}
\end{equation}
The projection $\pi_\omm:Z_\omm\to Z_{\omm'}$ 
has a section $\iota$, such that 
$f_\omm\circ\iota=f_{\omm'}$.
The relative tangent bundle for $\pi_k$ is denoted by $L_k$. It is associated  with the $\T$ representation of weight $-\alpha_k$, see \cite{Ram1}, \cite[\S3]{OSWW}. According to our notation given in \eqref{def:LL}$$L_k=\LL{\alpha_k}.$$ 

\subsection{Fixed points of the Bott-Samelson varieties}\label{sec:fixZ}

Let $\omm=s_{k_1}s_{k_2}\dots s_{k_\ell}$ be a reduced word representing $\om\in W$ and let 
$f_\omm:Z_\omm\to X_{\om}$ be the Bott-Samelson resolution of the Schubert variety $X_{\om}$. The $\T$ fixed points of $X_\om$ and $Z_\omm$ are discrete, namely:
\begin{itemize}
\item The fixed points $(X_{\om})^\T$ are $x_{\om'}$ where $\om'\leq \om$ in the Bruhat order. 
\item The fixed points $(Z_\omm)^\T$ are indexed by subwords of $\omm$ (which are words obtained by leaving out some of the letters from $\omm$). We identify subwords with 
01-sequences (where $0$'s mark the positions of the letters to be omitted). We will identify a fixed point with its subword and with its 01-sequence. 
\end{itemize}
The map $f_\omm$ sends the sequence $(\epsilon_1,\epsilon_2,\dots,\epsilon_\ell)$ to $x_\sigma$ where
$\sigma=s_{k_1}^{\epsilon_1}s_{k_2}^{\epsilon_2}\dots s_{k_\ell}^{\epsilon_\ell}\in W$.

\begin{example}\rm \label{ex:table}
Let $G=\GL_3$ and $\omm=s_1s_2s_1$. 
\begin{center}
\def\pp{\phantom{aa}}
\begin{tabular}{ |c|c|r| } 
 \hline
sequence&subword&image\pp\pp\\
\hline  
000 & $\cancel{s_1}\cancel{s_2}\cancel{s_1}$ & $\id=123$\pp \\ 
 001 & $\cancel{s_1}\cancel{s_2}s_1$ & $s_1=213$\pp\\
 010 & $\cancel{s_1}s_2\cancel{s_1}$ & $s_2=132$\pp \\
 011 & $\cancel{s_1}s_2s_1$ & $s_2s_1=312$\pp \\
 100 & $s_1\cancel{s_2}\cancel{s_1}$ & $s_1=213$\pp \\
 101 & $s_1\cancel{s_2}s_1$ & $\id=123$\pp \\
 110 & $s_1s_2\cancel{s_1}$ & $s_1s_2=231$\pp \\
 111 & $s_1s_2s_1$ & $s_1s_2s_1=321$\pp \\
 \hline
\end{tabular}
\end{center}
\end{example}

From the recursive definition of $Z_{\omm}$ above we find the recursive description of the tangent characters of $Z_\omm$ at the fixed point $x\in \{0,1\}^\ell$:
\[
\characters(Z_\omm,x)=
\begin{cases}
\characters(Z_{\omm'},x')\cup \{(L_k)_{f_{\omm'}(x')}\} & \text{ if } x=(x',0) \\
\characters(Z_{\omm'},x')\cup \{(L^{-1}_k)_{f_{\omm'}(x')}\} & \text{ if } x=(x',1).
\end{cases}
\]
Note that $L_k$ or $L_k^{-1}$ at a fixed point is just a line with $\T$ action, i.e.~a character, so it can indeed be interpreted as a function on $\T$ or an element of $\RR(T)$. In theory characters form a multiset, so above the $\cup$ should mean union of multisets, but in fact no repetition of characters occurs, hence there is no need for multisets.

\subsection{Canonical divisors of Schubert and Bott-Samelson varieties}
The starting point for the computation of the elliptic classes of Schubert varieties is the following fact.
Let $f_\omm:Z_\omm\to X_\om$ be the Bott-Samelson resolution associated to the word $\omm$. The subvariety $\partial Z_\omm=f_\omm^{-1}(\partial X_\om)$ with the reduced structure is of codimension one. Its components correspond to the subwords with one letter omitted:  
$$\partial Z_\omm=\bigcup_{i=1}^{\ell(\om)}\partial_i Z_\omm.$$
We consider $\partial Z_\omm$ as a divisor, the sum of components with the coefficients equal to one. The last component $\partial_{\ell(\om)}Z$ is the image of the section $\iota$ in the diagram \eqref{BS-diagram}.
\begin{theorem}  \label{fullflagcnonical}
The divisor $K_{X_{\om}}+\partial X_{\om}$ is a Cartier divisor, and
we have 
\[
f_{\omm}^*(K_{X_{\om}}+\partial X_{\om})=K_{Z_\omm}+\partial Z_\omm.
\]
\end{theorem}

 
We would like to  stress the importance of this theorem. It allows us to pull-back the divisor $K_{X_\om}+\partial X_\om$ to any resolution, and to use intersection product in order to compute characteristic classes. In addition the pull back to the Bott-Samelson resolution has strikingly simple form: all coefficients of the divisor $D=f_\omm^*(K_{X_\om}+\partial X_\om)-K_{Z_\omm}$ are equal to one.
 
\begin{proof} 
Let $\rho\in\t^*$ be half of the sum of positive roots. Let $\CC_{-\rho}$ be the trivial bundle with the  $\T$ action of weight $-\rho$ and let $\LL{\rho}=G\times_B \CC_{-\rho}$ be the line bundle associated with weight $\rho$. We denote ideal sheaves by $I(\ )$, and canonical sheaves by $\oomega_\bullet$. We have the following identities on equivariant sheaves: 
\begin{enumerate}
\item $\oomega_{X_\om}=I(\partial X_\om) \otimes \LL{\rho}^{-1} \otimes \C_{-\rho}$,
\item $\oomega_{Z_\omm}=I(\partial Z_\omm) \otimes f_{\omm}^*(\LL{\rho}^{-1}) \otimes \C_{-\rho}$.
\end{enumerate}
The first one is proved in \cite[Prop. 2.2]{GrKu} (cf. the non-equivariant version \cite[Th~4.2]{Ram2}), and the second one is proved in \cite[Prop.~2.2.2]{BrKu} (cf. the non-equivariant version \cite[Prop.~2]{Ram1}).

The boundary $\tilde{D}$ of the {\em opposite} open Schubert cell (that is $\tilde{D}=\cup X^{s_i}$) intersects $X_\om$ transversally, hence, from the sheaf identities above we obtain the divisor identities
\[
K_{X_\om}=-\partial X_\om - \tilde{D}\cap X_{\om},
\qquad
K_{Z_\omm}=-\partial Z_\omm - f^{-1}(\tilde{D}\cap X_{\om}).
\]
Hence, $K_{X_\om}+\partial X_\om$ is Cartier and, by rearrangement we obtain $K_{Z_\omm}+\partial {Z_\omm}=f^*( K_{X_\om}+\partial X_\om)$.
Note, that all the involved Weil divisors are $\T$ invariant.
\end{proof}

\subsection{The $\lambda$-twisted elliptic class of Schubert varieties}
\label{sec:twisted}

Let $\lambda\in \t^*$ be a strictly dominant weight (i.e.~belonging to the interior of the Weyl chamber) and let  
$\LL{\lambda}=G\times_B\CC_{-\lambda}$ be the associated (globally generated) line bundle over $G/B$. Then $H^0(G/B;\LL{\lambda})$ is the irreducible representation of $G$ with highest weight $\lambda$. There exists a unique (up to a constant) section $\xi_\lambda$ of $\LL{\lambda}$, which is invariant with respect to the nilpotent group $N^-\subset B^-$ and on which $\T$ acts via the character $\lambda$. Therefore $\xi_\lambda$ does not vanish at the points of the open Schubert cell $X_{\om_0}$ and its zero divisor is supported on the union of codimension one Schubert varieties. The translation $\xi_{\lambda}^\om:= \tilde \om\tilde \om_0^{-1}(\xi_\lambda)$ of this section by $\tilde w\tilde w_0^{-1}\in N(T)$ is an eigenvector of $B^-$ of weight $\om\lambda$. The zero divisor of ${\xi^\om_\lambda}_{|X_{\om}}$ is supported on $\partial X_{\om}$. The multiplicities of this zero divisor are given by the Chevalley formula. Namely, if $\om=\om' s_\alpha $, $\ell(\om)=\ell(\om')+1$ and $\alpha$ is a positive root, then the multiplicity of $X_{\om'}$ is equal to $\langle\lambda,\alpha^\vme\rangle$, where $\alpha^\vme$ is the dual root, \cite{Che} (or see \cite[Prop. 1.4.5]{BrionFlag} for the case of $\GL_n$).

\begin{example}\rm Let $G=\GL_n$, 
\begin{itemize}
\item $\lambda=(\lambda_1\geq\lambda_2\geq\dots\geq\lambda_n)$,
\item $\alpha=(0,\dots,1,\dots,-1,\dots,0)$, $\alpha^\vme=\lambda^*_i-\lambda^*_j$ for $i<j$, 
\end{itemize}
then 
$\langle\lambda,\alpha^\vme\rangle=\lambda_i-\lambda_j$. 
\end{example}

Our main object for the rest of the paper is the $\lambda$-twisted elliptic class
\begin{align*}
\Ek(X_{\om},\lambda)
=  \ELLt(X_\om, \partial X_\om-zeros(\xi_\lambda^\om);G/B),
\end{align*} 
of the Schubert variety $X_\om$, cf. Section \ref{sec:twisted}. The definition makes sense for ``sufficiently large'' $\lambda$, i.e. we need to assume that the coefficients of each boundary component in $zeros(\xi_\lambda^\om)$ is positive. It will be clear from the next section, that it is enough to assume that $\lambda$ is strictly dominant.

\section{Recursive calculation of local elliptic classes}

After using Kempf's lemma to calculate the multiplicities of $f_\omm^*(\xi^\om_\lambda)$ in Section \ref{sec:mults}, we will give a recursive formula for the local elliptic classes of the Bott-Samelson and Schubert varieties in Sections~\ref{sec:recursionZ},~\ref{sec:recursionX}.

\subsection{Multiplicities of the canonical section}\label{sec:mults}

As before, let $\omm$ be a reduced word representing $\om$, and consider the corresponding Bott-Samelson resolution. We have $f_{\omm}^*(K_{X_{\om}}+\partial X_{\om})=K_{Z_\omm}+\partial Z_\omm$ (Theorem \ref{fullflagcnonical}), and hence
\[
\Ek(X_{\om},\lambda) = f_{\omm*}\ELLt(Z_\omm,\partial Z_\omm-zeros(f_\omm^*(\xi_\lambda^\om))).
\]
Thus, to calculate $\Ek(X_{\om},\lambda)$ we need to know the multiplicities of $f_\omm^*(\xi^\om_\lambda)$ along the components of $\partial Z_\omm$.  Recall that the components of $\partial Z_\omm$ correspond to omitting a letter in the word $\omm$. Let $\partial_jZ_\omm$ denote the component corresponding to omitting the $j$'th letter of $\omm$. 
For our argument in the next section we need the following corollary drown from a sequence of papers dealing with Bott-Samelson resolutions.
\begin{proposition}\label{corinductive}Suppose $\lambda\in \t^*$ (not  necessarily dominant), then
\begin{enumerate}
\item $f^*_\omm(\LL{\lambda})\simeq \pi_\omm^*({\LL{s_k\lambda}}_{|Z_{\omm'}})\otimes\OO_{Z_\omm}(\langle \lambda,\alpha_k^\vme\rangle \iota Z_{\omm'})$,
\item if $\lambda$ is dominant, then the multiplicity of zeros of $f_\omm^*\xi^\om_\lambda$ along the divisor $\iota Z_{\omm'}$ is equal to $\langle\lambda, \alpha_k^\vme\rangle$,
\item the remaining multiplicities of $f^*_\omm \xi^\om_\lambda$ along the components of $\partial Z_\omm$ are equal to the corresponding multiplicities $f_{\omm'}^*\xi^{\om'}_{s_k\lambda}$.
\end{enumerate}
\end{proposition}

\begin{proof}  Suppose $\omm=\omm's_k$ is a reduced word. Let $Y=X_{\om'}\times_{G/P}G/B$.  We have the commutative diagram
\[\xymatrixcolsep{3pc}
\xymatrix{Z_\omm\ar[d]^{\pi_{\omm}}\ar[r]{}\ar@/^1pc/[rr]^{f_\omm}&
Y\ar[d]^{\pi_Y}\ar[r]_{f_Y\phantom{xxx}}& G/B \ar[d]^{\pi_k}\\
Z_{\omm'}
\ar[r]^{ f_{\omm'}\phantom{aa}}&X_{\om'}
\ar[r]^{}& G/P_k, \\}
\]
together with a section $\iota':X_{\omm'}\to Y$ which agrees with the section $\iota:Z_{\omm'}\to Z_\omm$. Recall the following lemma of Kempf (originating from the Chevalley's work).

\begin{lemma}[\cite{Kem} Lemma 3]\label{kempf} Suppose $\lambda\in \t^*$ (not  necessarily dominant), then
\begin{enumerate}[(i)]
\item $f^*_Y(\LL{\lambda})\simeq \pi_Y^*({\LL{s_k\lambda}}_{|X_{\om'}})\otimes\OO_Y(\langle \lambda,\alpha_k^\vme\rangle \iota'X_{\om'})$.
\item If $f_Y^*(\LL{\lambda})$ has a non-zero section, then so does $\pi_Y^*({\LL{s_k\lambda}}_{|X_{\om'}})$. Furthermore, there exists a non-zero section, which is invariant with respect to $B^-$. 
\end{enumerate}
\end{lemma}

To continue the proof note that 
(1) follows directly from (i).	
The bundle $f_\omm^*( \LL{\lambda})$ is isomorphic to $\OO_{Z_\omm}(\sum_{k=1}^{\ell(\om)} a_k \partial_kZ_\omm)$. 
The section in (ii) is unique and it is equal to $f_Y^*\xi^\omega_\lambda$.
The multiplicity corresponding to the last component $\partial_{\ell(\om)}Z_\omm=\iota(Z_{\omm'})$ is equal to $\langle \lambda,\alpha_k^\vme\rangle$, the remaining components are pulled back from $Z_{\omm'}$. This proves (2) and (3).
\end{proof}

\noindent It is immediate to show inductively the Chevalley formula for the multiplicities in the Bott-Samelson variety.

\begin{corollary}\label{lastword}
Let $\omm=\omm_1s_{k_j}\omm_2$ where $s_{k_j}$ is the $j$'th letter in the word $\omm$. 
Write $\om= \om_1 \om_2s_\alpha$, where $s_\alpha=\om_2^{-1}s_{k_j}\om_2$ and $\alpha$ is a positive root. Then 
the multiplicity of $f_\omm^*(\xi^\om_\lambda)$ along the boundary component $\partial_j Z_\omm$ is $\langle\lambda,\alpha^\vme\rangle$.\qed
\end{corollary}

\subsection{Local elliptic classes of Bott-Samelson and Schubert varieties} \label{sec:localBSS}
Recall from Section~\ref{sec:fixZ} that the $\T$ fixed points of $G/B$ are identified with elements of $W$, and the $\T$ fixed points of $Z_\omm$ are parameterized by subwords of $\omm$ or equivalently by 01-sequences. For $\sigma \in (G/B)^{\T}$, $x\in Z_\omm^{\T}$ define the local classes
\begin{align*}
E_\sigma(X_\omm,\lambda)=&\frac{\Ek(X_\omm,\lambda){|\sigma}}{\eekk(T_\sigma (G/B))}=\frac{  \ELLt(X_\om, \partial X_\om-zeros(\xi^w_\lambda)){|\sigma}}{\eekk(T_\sigma (G/B))}  \ \ \ \ \ \in Frac(\RR(\T)[h^{1/N}][[q]]),\\
E_{x}(Z_\omm,\lambda)=&\frac{  \ELLt(Z_\omm,\partial Z_\omm-zeros(f_\omm^*(\xi^w_\lambda))){|x} }{\eekk(T_{x}Z_\omm)}  \ \ \ \ \   \ \ \ \ \   \in Frac(\RR(\T)[h^{1/N}][[q]])
\end{align*}
in the fraction field of the representation ring $\RR(\T)$ extended by  formal formal roots of the parameter $h$.
If $\om=\id\in W$ then $\omm=\emptyset$. The Bott-Samelson variety $Z_\emptyset$ is one point, a fixed point indexed by the sequence of length 0. Hence we have $f_\emptyset(Z_\emptyset)=[\id]\in G/B$ and 
\begin{equation}\label{eq:base}
E_{\emptyset}(Z_\emptyset,\lambda)=1, \qquad E_{[\id]}(X_{\id},\lambda)=1.
\end{equation}
In the next two subsections we show how the geometry described in Section \ref{BoSa}  implies recursions of the local classes, that together with the base step \eqref{eq:base} determine them. 

\subsection{Recursion for local elliptic classes of Bott-Samelson varieties}\label{sec:recursionZ}

From the description of the fixed points of $Z_\omm$ and their tangent characters in Section \ref{sec:fixZ} we obtain the following recursion for the local classes on $Z_\omm$. Let $\omm=\omm' s_k$ be a reduced word. 

If $x=(x',0)\in Z^\T_\omm$ then
\[
E_x(Z_\omm,\lambda)= E_{x'}(Z_{\omm'},s^\lambda_k\lambda) \cdot \delta\left((L_k)_{f_{\omm'}(x')},h^{\langle\lambda,\alpha_k^\vme\rangle}\right).
\]

If $x=(x',1)\in Z^\T_\omm$ then
\[
E_x(Z_\omm,\lambda)=E_{x'}(Z_{\omm'},s^\lambda_k\lambda)\cdot \delta\left((L_k^{-1})_{f_{\omm'}(x')},h\right).
\]
As before, in our notation we identify a bundle restricted to a fixed point with the character of the obtained $\T$ representation on $\C$. Note that in the formula for $x=(x',1)$ the character $(L_k^{-1})_{f_{\omm'}(x')}$ is equal to $(L_k)_{f_{\omm}(x)}$.

Recall that the classes $E_x(Z_\omm,\lambda)$ are defined only for strictly dominant weights $\lambda$. However, the formulas above are meromorphic functions in $\lambda$ (with poles on the hyperplanes $z\langle \lambda, \alpha_k^\vme\rangle\in \ZZ$), so we formally define $E_x(Z_\omm,\lambda)$ for all $\lambda\in \t^*$ by the meromorphic function it satisfies for strictly dominant~$\lambda$. 

\begin{example} \rm 
For $G=\GL_3$ we use the notation $\lambda=(\lambda_1,\lambda_2,\lambda_3)$ as before, and let $\mu_k=\mu_k(\lambda)=h^{-\lambda_k}$. 
\begin{itemize}
\item For $\omm=\emptyset$ we have $E_\emptyset(Z_\omm,\lambda)=1$.
\item For $\omm=s_1$ we have 
\begin{align*}
E_{(0)}(Z_{s_1},\lambda)=& 1_{|\mu_1 \leftrightarrow \mu_2} \cdot \delta(z_2/z_1,\mu_2/\mu_1)=\delta(z_2/z_1,\mu_2/\mu_1) \\
E_{(1)}(Z_{s_1},\lambda)=& 1_{|\mu_1 \leftrightarrow \mu_2} \cdot \delta(z_1/z_2,h)=\delta(z_1/z_2,h).
\end{align*}
\item For $\omm=s_2s_1$ we have 
\begin{align*}
E_{(0,0)}(Z_{s_1s_2},\lambda)=& E_{(0)}(Z_{s_1},\lambda)_{|\mu_2 \leftrightarrow \mu_3} \cdot \delta(z_3/z_2,\mu_3/\mu_2) \\
                                            =& \delta(z_2/z_1,\mu_3/\mu_1)\delta(z_3/z_2,\mu_3/\mu_2), \\
E_{(0,1)}(Z_{s_1s_2},\lambda)=& E_{(0)}(Z_{s_1},\lambda)_{|\mu_ 2\leftrightarrow \mu_3} \cdot \delta(z_2/z_3,h) \\
                                            =& \delta(z_2/z_1,\mu_3/\mu_1)\delta(z_2/z_3,h), \\
E_{(1,0)}(Z_{s_1s_2},\lambda)=& E_{(1)}(Z_{s_1},\lambda)_{|\mu_2 \leftrightarrow \mu_3} \cdot \delta(z_3/z_1,\mu_3/\mu_2) \\
                                            =& \delta(z_1/z_2,h)\delta(z_3/z_1,\mu_3/\mu_2), \\
E_{(1,1)}(Z_{s_1s_2},\lambda)=& E_{(1)}(Z_{s_1},\lambda)_{|\mu_ 2\leftrightarrow \mu_3} \cdot \delta(z_1/z_3,h)\\
                                            =&  \delta(z_1/z_2,h)\delta(z_1/z_3,h).
\end{align*}
\end{itemize}
\end{example}

\subsection{Recursion for local elliptic classes of Schubert varieties}\label{sec:recursionX} 
Let $\omm$ represent $\om\in W$. According to Proposition \ref{prop:localization_general} we have
\begin{equation}\label{eq:Esum}
E_\sigma(X_\om,\lambda)=\sum_{x} E_{x}(Z_\omm,\lambda),
\end{equation}
where the summation runs for fixed points $x$ with $f_\omm(x)=\sigma$. 

For example, for $G=\GL_3$ let $\omm=s_1s_2s_1$ and let $x$ be the fixed point corresponding to the identity permutation. Then the summation has two summands, corresponding to the 01 sequences (fixed points) $(0,0,0)$ and $(1,0,1)$. In fact for this $\omm$ only two fixed points ($\id$ and $s_1$) are such that the summation has two terms; in the remaining cases there is only one term, see the table in Example \ref{ex:table}.

The recursion of Section \ref{sec:recursionZ} for the terms of the right hand side of \eqref{eq:Esum} implies a recursion for the $E_\sigma(X_\om,\lambda)$ classes: the initial step is
\[
E_\sigma(X_{id},\lambda)=\begin{cases}1&\text{if } \sigma=\id\\0&\text{if } \sigma\neq id,\end{cases}
\]
and for $\om=\om' s_k$, $\ell(\om)=\ell(\om')+1$ we have
\begin{equation}\label{eq:Xrecursion}
E_\sigma(X_{\om},\lambda)= E_\sigma(X_{\om'},s_k(\lambda)) \cdot \delta(L_{k,\sigma},h^{\langle\lambda,\alpha_k^\vme\rangle}) 
+  E_{\sigma s_k}(X_{\om'},s_k(\lambda)) \cdot \delta(L_{k,\sigma},h).
\end{equation}

\begin{example} \label{GL3recu}\rm
Let $G=\GL_3$ and assume we already calculated the fixed point restrictions of $\Ek(X_{312},\lambda)$. By applying  the recursion above for $\om=\om's_2=(s_2s_1)s_2$ we obtain
\begin{align}
\notag E_{123}(X_{321},\lambda)=&
E_{123}(X_{312},\lambda)_{|\mu_2\leftrightarrow \mu_3}\delta(z_3/z_2,\mu_3/\mu_2)+
E_{213}(X_{312},\lambda)_{|\mu_2\leftrightarrow \mu_3}\delta(z_3/z_2,h)
\\ \label{4term1}
 =&
\delta(z_3/z_2,\mu_2/\mu_1)\delta(z_2/z_1,\mu_3/\mu_1)\delta(z_3/z_2,\mu_3/\mu_2) \\
& \ \hskip 4 true cm +\delta(z_2/z_3,h)\delta(z_3/z_1,\mu_3/\mu_1)\delta(z_3/z_2,h). \notag
\end{align}
We may calculate the same local class using the recursion for $\om=\om's_1=(s_1s_2)s_1$, and we obtain
\begin{align}  \label{4term2}
 E_{123}(X_{321},\lambda)=&
\delta(z_2/z_1,\mu_3/\mu_2)\delta(z_3/z_2,\mu_3/\mu_1)\delta(z_2/z_1,\mu_2/\mu_1) \\
& \ \hskip 4 true cm +\delta(z_1/z_2,h)\delta(z_3/z_1,\mu_3/\mu_1)\delta(z_2/z_1,h). \notag
\end{align}
The equality of the expressions \eqref{4term1} and  \eqref{4term2} is a non-trivial four term identity for theta functions,
for more details see Section \ref{eg:FL3}.
\end{example}

\medskip

Now we are going to rephrase the recursion \eqref{eq:Xrecursion} in a different language. According to K theoretic equivariant localization theory, the map
\[
K_\T(G/B) \to \bigoplus_{\sigma \in W} K_\T(x_\sigma)=\bigoplus_{\sigma \in W} \RR(\T),
\]
whose coordinates are the restriction maps to $x_\sigma$, is injective. As before (cf. Sections \ref{sec:localization}, \ref{sec:localBSS})  
we divide the restriction by the K theoretic Euler class, and consider the map
\[
\res: K_\T(G/B)\to \bigoplus_{\sigma \in W} Frac(\RR(\T)), \qquad\qquad 
\res(\beta)=\left\{ \frac{ \beta_{|x_\sigma}}{\eekk(T_{x_\sigma}(G/B))} \right\}_{\sigma \in W}
\]
where $Frac(\RR(\T))$ is the fraction field of the representation ring $\RR(\T)$. Since the map is injective, we may identify an element of $K_\T(G/B)$ with its $\res$-image, i.e. with a tuple of elements from of $Frac(\RR(\T))$.
Let us note that $E_\sigma(X_\om,\lambda)$ is expressed by factors $\delta(z,h)$ and $\delta(z,h^{\langle\lambda, \alpha^\vme\rangle})$ where $z$ is a character of the maximal torus and $\alpha^\vme$ is a coroot. Let $E_\sigma(X_\om)$ denote the associated function on $\lambda\in \t^*$. 
In the basis of simple coroots $\{\alpha^\vme_i\}$ the function $E_\sigma(X_\om)$ can be expressed by the functions $\mu_i(\lambda)=h^{-\langle \lambda,\alpha^\vme\rangle}=\eee^{ z\langle \lambda,\alpha^\vme\rangle}$. 
The variables $\mu_i$ are functions on the torus $\T^*=\t^*/z\t^*_\ZZ$.
Therefore we can treat the restricted classes as elements of the following objects
$$E_\sigma(X_\om)\in Frac(\RR(\T\times \T^*\times \C^*))[[q]]\,.$$

\begin{theorem}[Main Theorem]\label{th:mainiduction} Regarding the classes $E_\bullet(X_\om)$ as elements of $$\mathcal M=\oplus_{\sigma \in W} Frac\big(\RR({\T\times \T^*\times \C^*})\big)[[q]]\,,$$ the following recursion holds:
\begin{equation}\label{Eini}
E_\sigma(X_{\id},\lambda)=\begin{cases} 1 & \sigma=\id \\ 0 & \sigma\not=\id; \end{cases}
\end{equation}
and for $\om=\om's_k$ with $\ell(\om)=\ell(\om')+1$ we have 
\[
E_\bullet(X_{\om},\lambda)=(\delta^{bd}_k\id+\delta^{int}_ks^\gamma_k) \left(E_\bullet(X_{\om'},s_k\lambda)\right).
\]
Here 
\begin{itemize}
\item $s_k^\gamma$ for $s_k\in W$ acts on the fixed points by right translation $\sigma\mapsto \sigma s_k$,
\item $s_k$ acts  on $\lambda\in\t^*$, later this action will be denoted by $s^{\mu}_k$,
\item $\delta^{bd}_k$ --- multiplication by the element $\delta(L_k,h^{\langle\alpha_k^\vme,\lambda\rangle})$ (the ``boundary factor'')
\item $\delta^{int}_k$ --- multiplication by the element $\delta(L_k,h)$ (the ``internal factor'').
\end{itemize}
\end{theorem}

Note that the boundary and internal factors indeed make sense: restricted to a fixed point $x_\sigma$ the line bundle $L_k$ is a $\T$ character depending on $\sigma$; that  is, multiplication by one of these factors means multiplication by a diagonal matrix, not by a constant matrix.

\begin{proof} The statement is the rewriting of the recursion \eqref{eq:Xrecursion}. \end{proof}

\section{Hecke algebras}\label{sec:Hecke}

In this section we review various Hecke-type actions on cohomology or K-theory of $G/B$ giving rise to inductive formulas for various invariants of the Schubert cells.
Sections \ref{sec:HeckeNil}--\ref{sec:HeckeMC}---exploring the relation between our elliptic classes and other characteristic classes of singular varieties---is not necessary for the rest of the paper. A reader not familiar with Chern-Schwartz-MacPherson or motivic Chern classes can jump to Section \ref{sec:HeckeEll}. 

\subsection{Fundamental classes---the nil-Hecke algebra} \label{sec:HeckeNil}

Consider the notion of {\em equivariant fundamental class in cohomology}, denoted by $[\ ]$. According to \cite{BGG, Dem} if $\om=\om's_k$, $\ell(w)=\ell(\om')+1$ then the Demazure operation
$D_k=\pi^*_k\circ {\pi_k}_*$ in cohomology satisfies
$$D_k([X_{\om'}])=[X_{\om}]\,, \qquad D_k\circ D_k=0.$$
The algebra generated by the operations $D_k$ is called the nil-Hecke algebra.
As before, let us identify elements of $H^*_\T(G/B)$ with their $\res$-image, where 
\[
\res:H^*_\T(G/B)\to \bigoplus_{\sigma\in W} \QQ(\t), \qquad \qquad 
\beta \mapsto \left\{\frac{\beta_{|x_\sigma}}{\eehh(T_{x_\sigma}(G/B))}\right\}_{\sigma\in W}.
\]
Here $\QQ(\t)$ is the field of rational functions on $\t$, and $\eehh(\ )$ is the equivariant {\em cohomological} Euler class.

The action of the Demazure operations on the right hand side is given by the formula
\[
D_k=\tfrac 1{c_1(L_k)}(\id+s^\gamma_k),\qquad \text{that is} \qquad
D_k(\{f_\bullet\})_\sigma=
 \tfrac 1{c_1(L_k)}_{\sigma}(f_\sigma+f_{\sigma s_k}).\]
The operators $D_k$ satisfy the braid relations and $D_k\circ D_k=0$.

For $G=\GL_n$ we have $c_1(L_k)_{\sigma}=z_{\sigma(k+1)}-z_{\sigma(k)}$
(where $z_1,z_2,\dots,z_n$ are the basic weights of $\T\leq \GL_n$) and we recover the divided difference operators from algebraic combinatorics.

\subsection{CSM-classes and the group ring $\ZZ[W]$} An important one-parameter deformation of the notion of cohomological (equivariant) fundamental class is the equivariant Chern-Schwartz-MacPherson (CSM, in notation $c^{sm}(-)$) class. For introduction to this cohomological characteristic class see, e.g. \cite{Oh, WeCSM, AlMi, FR, AMSS}.

It is shown in \cite{AlMi, AMSS} that the CSM classes of Schubert cells satisfy the recursion:  if $\om=\om's_k$, $\ell(w)=\ell(\om')+1$ then
$$A_k(c^{sm}(X^\circ_{\om'}))=c^{sm}(X^\circ_{\om})$$
where
$$A_k=(1+c_1(L_k))D_k-\id\,.$$
In terms of $\res$-images  
$$A_k(\{f_\bullet\})_\sigma=
\tfrac 1{c_1(L_k)_{\sigma}} f_\sigma+ \tfrac {1+c_1(L_k)_{\sigma}}{c_1(L_k)_{\sigma}} f_{\sigma s_k}.$$ 
By \cite{AlMi} or by straightforward calculation we find that $A_k\circ A_k=\id$ and the operators $A_k$ satisfy the braid relations.

\subsection{Motivic Chern classes---the Hecke algebra}\label{sec:HeckeMC}
The K theoretic counterpart of the notion of CSM class is the {\em motivic Chern class} (in notation $mC_y(-)$), see \cite{BSY, FRW, AMSS2}. 
The operators 
\[
B_k=(1+y L_k^{-1})\pi_k^*{\pi_k}_*-\id\;\in\;End(K_\T(G/B)[y]), .
\]
(see Section \ref{sec:BSres}) reproduce the motivic Chern classes $mC_y$ of the Schubert cells:  if $\om=\om's_k$, $\ell(w)=\ell(\om')+1$ then
$B_k(mC_y(X^\circ_{\om'}))=mC_y(X^\circ_{\om})$, see \cite{AMSS2}, c.f.~\cite{SZZ}. 
In the local presentation, i.e. after restriction to the fixed points and division by the K-theoretic Euler class,  
the operator $B_k$ takes form
$$B_k(\{f_\bullet\})_{\sigma}= 
 \frac{(1+y)(L_k^{-1})_{\sigma}}{1-(L_k^{-1})_{\sigma}} f_{\sigma}
+ \frac{1+y(L_k^{-1})_{\sigma}}{1-(L_k^{-1})_{\sigma}} f_{\sigma s_k}.$$
For example, for $G=\GL_n$ we have $(L_k)^{-1}_\sigma={z_{\sigma(k)}}/{z_{\sigma(k+1)}}$.
The squares of the operators satisfy 
$$B_k\circ B_k=-(y+1) B_k-y\, id,$$ 
and  the operators $B_k$ satisfy the braid relations. This kind of algebra was discovered much earlier by Lusztig \cite{Lusztig}.

\subsection{Elliptic Hecke algebra} \label{sec:HeckeEll}
Consider the operator 
$$C_k=(\delta^{bd}_k\id+\delta^{int}_ks^\gamma_k)s^{\mu}_k$$
acting on the direct sum of the spaces of rational functions extended by the formal parameter $q$ $$\oplus_{\sigma \in W} Frac\big(\RR(\T\times \T^*\times\C^*)\big)[[q]]\,,$$  or in coordinates
$$C_k(\{f_\bullet\})_\sigma(\lambda)=(\delta_k^{bd})_{\sigma}\, f_\sigma(s_k\lambda) +(\delta_k^{int})_{\sigma}\,f_{\sigma s_k}(s_k\lambda)\,.$$
In Section \ref{sec:recursionX} we have shown that  if $\om=\om's_k$, $\ell(w)=\ell(\om')+1$ then
$$E_\bullet(X_{\om},\lambda)=C_k \left(E_\bullet(X_{\om'},\lambda)\right).$$

\begin{theorem}\label{th:kappa}
The square of the operator $C_k$ is multiplication by a function depending only on $\lambda$ and $h$:
$$C_k\circ C_k=\kappa_k(\lambda)\,id\,,$$
where
$$\kappa_k(\lambda)=\delta(h,\nu_k(\lambda))\delta(h,1/\nu_k(\lambda))
\,,$$
where $\nu_k(\lambda)=h^{\langle\lambda,\alpha_k^\vme\rangle}$.
\end{theorem}
\begin{proof}
It is enough to check the identity for $G=\GL_2$, $\sigma=\id$ :
\begin{multline*}C_1\circ C_1(\{f_\bullet\})_{id}=\\
 \left(\delta\left(\frac{z_1}{z_2},h\right) \delta\left(\frac{z_2}{z_1},h\right)+\delta\left(\frac{z_2}{z_1},\frac{1}{\nu _1}\right) \delta\left(\frac{z_2}{z_1},\nu_1\right)\right)f_{id}+
\delta\left(\frac{z_2}{z_1},h\right) \left(\delta\left(\frac{z_1}{z_2},\frac{1}{\nu _1}\right)+\delta\left(\frac{z_2}{z_1},\nu_1\right)\right)f_{s_1 } \end{multline*}
Since the function $\theta$ is antisymmetric ($\vt(1/x)=-\vt(x)$ and hence $\delta(x,y)=-\delta(1/x,1/y)$) we have
$$C_1\circ C_1(\{f_\bullet\})_{id}=
 \left(\delta\left(\frac{z_1}{z_2},h\right) \delta\left(\frac{z_2}{z_1},h\right)+\delta\left(\frac{z_2}{z_1},\frac{1}{\nu _1}\right) \delta\left(\frac{z_2}{z_1},\nu_1\right)\right)f_{id}$$
$$= -\frac{\left(\vt
   (\nu _1)^2\,
   \vt \left(h\frac{
   z_1}{z_2}\right) \vt
   \left(h\frac{
   z_2}{z_1}\right)+\vt (h)^2\, \vt
   \left(\frac{z_2}{\nu _1
   z_1}\right) \vt
   \left(\frac{\nu _1
   z_2}{z_1}\right)\right)\vt'(1)^2}{\vt
   (h)^2\, \vt(\nu
   _1)^2 \,\vt
   \left(\frac{z_2}{z_1}\right)\!{}^2}f_{id}$$
Setting
$$a= h,\quad 
b= \frac{z_2}{z_1},\quad 
c= \nu_1,\quad 
d= 1$$
in the Fay's trisecant identity \eqref{trisecant} we obtain the claim.
\end{proof}
\medskip

\noindent{\it Proof of Formula \ref{BS-intro}} (from the Introduction) relating global nonrestricted classes, we conjugate the operation $C_k$ with the multiplication by the elliptic Euler class 
\begin{equation}\label{eq:ellipticEuler}\eell(TG/B)=\prod_{\alpha\,\in\,\text{positive roots}}\vt(\LL{\alpha})\end{equation}
(according to our convention $\LL{\alpha}=G\times_B\CC_{-\alpha}$). Since $$\vt(\LL{\alpha}|\sigma)/\vt(\LL{\alpha}|\sigma s_k)=-1$$ we obtain the minus sign in \eqref{BS-intro}.
\medskip

If $s_ks_\ell=s_\ell s_k$, then $C_k\circ C_\ell=C_\ell\circ  C_k$. 
Moreover for $\GL_n$
$$C_k\circ C_{k+1}\circ C_k=C_{k+1}\circ C_k\circ C_{k+1}.$$
Hence for $G=\GL_n$ the operators $C_k$ define a representation of the braid group.
For general $G$ the corresponding braid relation of the Coxeter group are satisfied. This is an immediate consequence of the fact, that the elliptic class does not depend on the resolution.

\begin{remark}\rm  
Note that the operations $C_\ell$ do not commute with
$\kappa_k(\lambda)$
but they satisfy
$\kappa_k(\lambda)\circ C_\ell=C_\ell\circ\kappa_k(s_\ell\lambda)$.
\end{remark}

This table summarizes the various forms of the Hecke algebras whose operators produce more and more general characteristic classes of Schubert varieties.

\begin{center}{\def\arraystretch{2}
\begin{tabular}{ |c|c|c| } 
 \hline
invariant&operation&square\\ 
\hline  
$[-]$ & $\tfrac 1{c_1(L_k)}(\id+ s_k^\gamma)$ & $D_k^2=0$ \\ 
 $c^{sm}$ & $\tfrac 1{c_1(L_k)} \id+ \tfrac {1+c_1(L_k)}{c_1(L_k)} s_k^\gamma$ & $A_k^2=\id$ \\ 
 $mC_y$ & $\frac{(1+y)L_k^{-1}}{1-L_k^{-1}} \id
+ \frac{1+y\,L_k^{-1}}{1-L_k^{-1}} s^\gamma_k$ & $B_k^2=-(y+1)B_k-y\,\id$ \\ 
& & or $(B_k+y)(B_k+1)=0$ \\
$E(-,\lambda)$ & $\; \; \delta(L_k,h^{\langle\lambda,\alpha_k^\vme\rangle})
\, s^\mu_k +\delta(L_k,h)s^\gamma_ks^\mu_k\; \; 
$ & $C_k^2=\kappa_k(\lambda)$. \\
  \hline
\end{tabular}}
\end{center}
\bigskip

\subsection{Modifying the degeneration of $E(-,\lambda)$ and $C_k$} 
The characteristic classes, $[-]$, $c^{sm}$, $mC_y$, $E(-,\lambda)$ are of increasing generality: an earlier in the list can be obtained from a latter in the list by formal manipulations. 
However, the limiting procedure of getting $mC_y$ from $E(-,\lambda)$ itself is not obvious. We describe this procedure below. As a result we obtain a family of $mC_y$-like classes, with only one of them the $mC_y$-class, as well as a family of  corresponding Hecke-type algebras.

The theta function has the limit property
$$\lim_{q\to 0}\vt(x) = x^{1/2}-x^{-1/2}.$$
It follows that 
$$\lim_{q\to 0}\delta(x,h)=\frac{1-1/(x h)}{(1-1/x)(1-1/h)}\,.$$
The motivic Chern class is the limit of the elliptic class when $q=\eee^{2\pi i \tau}\to0$.
The limit of the elliptic class of a pair is not we would expect: the limit of the boundary factor is equal to
$$\lim_{q\to 0}\delta(x,\nu_k(\lambda))=\frac{1-L_k^{-1}/\nu_k(\lambda)}{(1-L_k^{-1})(1-1/\nu_k(\lambda))}$$
where 
$\nu_k(\lambda)=h^{\langle\lambda,\alpha_k^\vme\rangle}$.
The limit Hecke algebra differs  from the Hecke algebra computing $mC_y$'s of Schubert cells. The limit classes depend on the parameter $\lambda\in\t^*$.
Calculation shows that in the limit when $q\to 0$ we obtain the operators
$$C^{q\to 0}_k=\left(\frac{1-L_k^{-1}/\nu_k(\lambda)}{(1-L_k^{-1})(1-1/\nu_k(\lambda))}\, \id+\frac{1 -L_k^{-1}/h}{(1-L_k^{-1})(1-1/h)}s_k^\gamma\right)\circ s^\mu_k$$
satisfying
$$C^{q\to 0}_k\circ C^{q\to 0}_k=\frac{1 -\nu_k(\lambda)/h}{(1-\nu_k(\lambda))(1-1/h)}\cdot \frac{1 -1/(\nu_k(\lambda)h)}{(1-1/\nu_k(\lambda))(1-1/h)} id.$$
\bigskip

Another method of passing to the limit, as in \cite{BoLi2}, is
when we first rescale $\lambda$ by $\log(q)/\log(h)$ (thus where in the formulas we had $h^{\langle \lambda,\alpha^\vme_k\rangle}$ we have $q^{\langle \lambda,\alpha^\vme_k\rangle}$ instead)
and then pass to the limit. 
For  $0 < Re(z) < 1$ , and $m \in \ZZ$ we have
$$\lim_{q\to 0}\frac{\vt(a q^{ m+z} )}{\vt(b q^{ m+z} )}
= (a/b)^{-m-1/2}\,.$$
Passing to the limit we obtain different factors:
$$
\lim_{q\to 0}\delta (a, q^{\langle \lambda,\alpha_k^\vme\rangle})=\lim_{q\to 0}\frac{\vt'(1)\vt(a q^{\langle \lambda,\alpha_k^\vme\rangle})}{\vt(a)\vt(q^{\langle \lambda,\alpha_k^\vme\rangle})}=\frac{a^{-b_k(\lambda)-1/2}}{a^{1/2}-a^{-1/2}}=\frac{a^{-b_k(\lambda)-1}}{1-a^{-1}}\,, 
$$
where $b_k(\lambda)$ is the integral part of $Re \langle \lambda,\alpha_k^\vme\rangle$, provided that $ \langle \lambda,\alpha_k^\vme\rangle\not\in \Z$.
The limit operation now takes form
$$\widetilde C^{q\to 0}_k=\left( \frac{L_k^{-b_k(\lambda)-1}}{1-L_k^{-1}} \, \id+\frac{1 -L_k^{-1}/h}{(1-L_k^{-1})(1-1/h)}s^\gamma_k\right)\circ s^\mu_k\,.$$
Setting $y=-h^{-1}$ we obtain a form resembling the operation $B_k$:
$$\frac{1}{1+y}\left( \frac{(1+y)L_k^{-b_k(\lambda)-1}}{1-L_k^{-1}} \, \id+\frac{1 +y\,L_k^{-1}}{(1-L_k^{-1})}s_k^\gamma\right)\circ s^\mu_k\,.$$
For weights $\lambda$ belonging to the dominant Weyl chamber, which are sufficiently close to 0 we obtain the operation $B_k$. But note that here still the limit operation is composed with the action of $s_k$ on $\lambda\in \t^*$. In general we obtain a version of ``motivic stringy invariant'' mentioned in \cite[\S11.2]{SchYo}.

\begin{remark}\rm
The operators $\widetilde C^{q\to 0}_k$ map the so-called {\em trigonometric weight functions} of \cite[Section 3.2]{RTV} into each other. These functions also depend on an extra {\em slope} or {\em alcove} parameter, where a region in a subset of $\t^*$ where the functions $b_k$ are constant.  
The resulting multiplier for $\widetilde C^{q \to 0}_k\circ \widetilde C^{q\to 0}_k$ equals
$$\lim_{q\to 0} \delta(h,q^{\langle\lambda,\alpha^\vme_k\rangle})\delta(h,q^{-\langle\lambda,\alpha^\vme_k\rangle})=
 \frac{h^{-b_k(\lambda)-1}}{(1-h^{-1})}\cdot \frac{h^{-b_k(-\lambda)-1}}{(1-h^{-1})}=
-\frac{h^{-1}}{(1-h^{-1})^2}=
-\frac{ y}{(1+y)^2}\,$$
(since $b_k(\lambda)+b_k(-\lambda)=-1$), which, remarkably, does not depend of $\lambda$. 
\end{remark}

\section{Weight functions
}\label{sec:EllWeight}

In this section we focus on type A Schubert calculus, and give a formula for the elliptic class of a Schubert variety in terms of natural generators in the K theory of $\F(n)=G/B$. This formula will coincide with the {\em weight function} defined in \cite{RTV} (based on earlier weight function definitions of Tarasov-Varchenko \cite{TV}, Felder-Rim\'anyi-Varchenko \cite{FRV18}, see also \cite{konno}). Weight functions play an important role in representation theory, quantum groups, KZ differential equations, and recently (in some situations) they were identified with stable envelopes in Okounkov's theory.

\smallskip

For a non-negative integer $n$ let $\F(n)$ be the full flag variety parametrizing chains of subspaces  $0=V_0\subset V_1\subset \ldots \subset V_n=\C^n$ with $\dim V_k=k$. We will consider the natural action of $\T=(\C^*)^n$ on $\F(n)$. The $\T$-equivariant tautological rank $k$ bundle (i.e. the one whose fiber is $V_k$) will be denoted by $\TTT^{(k)}$, and let $\GGG_k$ be the line bundle $\TTT^{(k)}/\TTT^{(k-1)}$. 

Let $\gamma_k$ be the class of $\GGG_k$ in $K_{\T}(\F(n))$ and let $t^{(k)}_1,\ldots,t^{(k)}_k$ be the \gr roots of $\TTT^{(k)}$ (i.e. $[\TTT^{(k)}]=t^{(k)}_1+\ldots+t^{(k)}_k$) for $k=1,\ldots,n$.
Let us rename $t^{(n)}_j=z_j$. 

It is well known that the $\T$-equivariant K ring of $\F(n)$ can be presented as
\begin{align}
\label{KT1}
K_{\T}(\F(n))&=
\Z[(t^{(k)}_a)^{\pm 1},z_j^{\pm 1}]_{k=1,\ldots,n-1,\ a=1,\ldots,k,\ j=1,\ldots,n}^{{S_1}\times \ldots \times S_{n-1}}/(\text{relations})\\
\label{KT2}
&= 
\Z[(\gamma_{k})^{\pm 1},z_j^{\pm1}]_{k=1,\ldots,n,\ j=1,\ldots,n}/(\text{relations}),
\end{align}
with certain relations.
The first presentation is a result of presenting the flag variety as a quotient of the quiver variety $$\F(n)=V/\!/G\,,\qquad V=\prod_{k=1}^{n-1}\Hom(\C^k,\C^{k+1})\,,\qquad G=\prod_{k=1}^{n-1}\GL_k,$$ see \cite[\S6]{FRW2}. Then
$$K_{G\times \T}(V)\;\twoheadrightarrow\; K_{G\times \T}(U)\simeq K_\T(\F(n))\,,$$
where $U$ is the open subset in $V$ consisting of the family of monomorphisms.
The variables $t^{(k)}_a$ are just the characters of the factor $\GL_k$.
The second presentation comes from a geometric picture as well:
$\F(n)=\GL_n/B_n$ is homotopy equivalent to $\GL_n/\T$ and
$$K_{\T\times \T}(\Hom(\C^n,\C^n))\;\twoheadrightarrow\; K_{\T\times \T}(\GL_n)\simeq K_\T(\GL_n/\T).$$
The variables $\gamma_k$ appearing in the presentation \eqref{KT2} are the characters of the second copy of $\T$ acting from the right on $\Hom(\C^n,\C^n)$.

Explicit generators of the ideal of relations could be named in both lines \eqref{KT1}, \eqref{KT2}, but it is more useful to understand the description of the ideals via ``equivariant localization'' (a.k.a. ``GKM description'', or ``moment map description''), as follows.

The $\T$ fixed points $x_\sigma$ of $\F(n)$ are parameterized by permutations $\sigma\in S_n$. 
The restriction map from $K_{\T}(\F(n))$ to $K_{\T}(x_\sigma)=\Z[z_j^{\pm 1}]_{j=1,\ldots,n}$ is given by the substitutions 
\begin{equation}\label{substitute}
t^{(k)}_a\mapsto z_{\sigma(a)}, \qquad\qquad\qquad 
\gamma_{k}\mapsto z_{\sigma(k)}.
\end{equation}
Symmetric functions in $t^{(k)}_a, z_j$, and functions in $\gamma_{k},z_j$ belong to the respective ideals of relations if and only if their substitutions \eqref{substitute} vanish for all $\sigma\in S_n$ \cite[Appendix]{KR}, \cite[Ch. 5-6]{ChGi}.

\medskip

Our main objects of study, the classes $\Ek(X_\om)$ live in the completion of $K_{\T}(\F(n))$ adjoined with variables $h,\mu_k$, that is, in the ring
\begin{align}
\label{ring1}
&\Z[[(t^{(k)}_a)^{\pm 1},z_j^{\pm 1},h,\mu_j^{\pm 1}]]_{k=1,\ldots,n-1,\ a=1,\ldots,k,\ j=1,\ldots,n}^{{S_1}\times \ldots \times S_{n-1}}/(\text{relations})\\
\label{ring2}
= &\Z[[(\gamma_{k})^{\pm 1},z_k^{\pm1},h,\mu_k^{\pm 1}]]_{k=1,\ldots,n}/(\text{relations}),
\end{align}
where the same localization description holds for the two ideals of relations. 
Our goal in this section is to define representatives---that is, functions in $t^{(k)}_a,z_j,h,\mu_j$ and functions in $\gamma_{k}, z_j, h,\mu_j$---that represent the elliptic classes of $\Ek(X_\om)$ of Schubert varieties. This goal will be achieved in Theorem~\ref{thm:EsameW} below.

\subsection{Elliptic weight functions}
Now we recall some special functions called elliptic weight functions, from \cite{RTV}.
For $\om\in S_n$, $k=1,\ldots,n-1$, $a=1,\ldots,k$ define the integers 
\begin{itemize}
\item $\om^{(k)}_a$ by  $\{\om(1),\ldots,\om(k)\}=\{\om^{(k)}_1<\ldots<\om^{(k)}_k\}$,
\item $j_\om(k,a)$ by $\om^{(k)}_a=\om(j_\om(k,a))$,
\item 
\[
c_\om(k,a)=\begin{cases}\hfill 0&\text{if }\om(k+1)\geq \om^{(k)}_a\\
1 &\text{if }  \om(k+1)<  \om^{(k)}_a.\end{cases}
\]

\end{itemize}

\begin{definition} \cite{RTV}
For $\om\in S_n$ define the elliptic weight function by
\[
\ww_\om= \left(\tfrac{\vt(h)}{\vt'(1)}\right)^{\dim G}\frac{\Sym_{t^{(1)}}\ldots\Sym_{t^{(N-1)}}  U_\om}{\prod_{k=1}^{n-1} \prod_{i,j=1}^{k}  \vt(ht^{(k)}_i/t^{(k)}_j)},
\]
where 
\[
G=\prod_{k=1}^{n-1}\GL_k\,,\qquad \dim G=\tfrac{(n-1)n(2n-1)}6,
\]
\[
\Sym_{t^{(k)}} f(t^{(k)}_1,\ldots,t^{(k)}_{k})=\sum_{\sigma\in S_{k}} f(t^{(k)}_{\sigma(1)},\ldots,t^{(k)}_{\sigma(k)}),
\]
\[
 U_\om=\prod_{k=1}^{n-1} \prod_{a=1}^{k}\left( 
 \prod_{c=1}^{k+1} \psi_{\om,k,a,c}(t^{(k+1)}_c/t^{(k)}_a) \prod_{b=a+1}^{k} 
  \delta(t^{(k)}_b/t^{(k)}_a,h)
 \right),
\]
\[
\psi_{\om,k,a,c}(x)
=\vt(x)\cdot\begin{cases}
   \delta(x,h) & \text{if}\qquad \om^{(k+1)}_c<\om^{(k)}_a \\
   \delta\left(x,h^{1-c_\om(k,a)} \frac{\mu_{k+1}}{\mu_{j_\om(k,a)}}\right) & \text{if}\qquad \om^{(k+1)}_c=\om^{(k)}_a \\
   1 & \text{if}\qquad \om^{(k+1)}_c>\om^{(k)}_a.
\end{cases}
\]
\end{definition}

\noindent The usual names of the variables of the elliptic weight function are:  
\[
\begin{array}{lll}
t^{(k)}_a &  \text{ for } k=1,\ldots,n-1,\ a=1,\ldots,k& \text{the topological variables,} \\
z_a:=t^{(n)}_a                    &  \text{ for } a=1,\ldots,n & \text{the equivariant variables,}\\
h & & \text{the ``Planck variable'',}\\
\mu_k=h^{-\lambda_k} & \text{ for } k=1,\ldots, n & \text{the dynamical (or K\"ahler) variables}. 
\end{array}
\]
The function $\ww_\om$ is symmetric in the $t^{(k)}_*$ variables (for each $k=1,\ldots,n-1$ separately), but not symmetric in the equivariant variables. 

Consider the new variables $\gamma_1,\ldots,\gamma_n$, and define the modified weight function (of the variables $\gamma=(\gamma_1,\ldots,\gamma_n)$, $z=(z_1,\ldots,z_n)$, $\mu=(\mu_1,\ldots,\mu_n)$, and $h$)
\begin{equation}\label{modW}
\wwh_\om(\gamma,z,\mu,h)=\ww_\om
\left(t^{(k)}_a=\gamma_a \text{ for } k=1,\ldots,n-1; t^{(n)}_a=z_a\right),
\end{equation}
that is, we substitute $\gamma_a$ for $t^{(k)}_a$ for $k=1,\ldots,n-1$, and rename $t^{(n)}_a$ to $z_a$. This substitution corresponds to going from the presentation \eqref{KT1} to the presentation \eqref{KT2}.

\begin{example} \rm
We have 
\[\ww_{12}=\tfrac{1}{\vt'(1)}{\vt\left(\tfrac{t^{(2)}_1}{t^{(1)}_1}\right)\vt\left(\tfrac{t^{(2)}_2}{t^{(1)}_1}\right)\delta\left(\tfrac{t^{(2)}_1}{t^{(1)}_1},\tfrac{h\mu_2}{\mu_1}\right)}
=
\tfrac{\vt\left(\tfrac{t^{(2)}_2}{t^{(1)}_1}\right)\vt\left(\tfrac{h\mu_2t^{(2)}_1}{\mu_1t^{(1)}_1}\right)}{\vt\left(\tfrac{h\mu_2}{\mu_1}\right)},\]
\[
\ww_{21}=
\tfrac{1}{\vt'(1)}\vt\left(\tfrac{t^{(2)}_1}{t^{(1)}_1}\right)\vt\left(\tfrac{t^{(2)}_2}{t^{(1)}_1}\right)\delta\left(\tfrac{t^{(2)}_1}{t^{(1)}_1},h\right)\delta\left(\tfrac{t^{(2)}_2}{t^{(1)}_1},\tfrac{\mu_2}{\mu_1}\right)=
\frac{\vt'(1)\vt
   \left(\tfrac{h
   t^{(2)}_1}{t^{(1)}_1}\right)
   \vt\left(\frac{\mu_2
   t^{(2)}_2}{\mu _1
   t^{(1)}_1}\right)}{\vt (h)
   \vt\left(\frac{\mu
   _2}{\mu _1}\right)},
\]
and hence 
\begin{equation}\label{WWex}
\wwh_{12}=\tfrac{1}{\vt'(1)}
\vt\left(\tfrac{z_1}{\gamma_1}\right)\vt\left(\tfrac{z_2}{\gamma_1}\right)\delta\left(\tfrac{z_1}{\gamma_1},\tfrac{h\mu_2}{\mu_1}\right),
\ \ \ \
\wwh_{21}=
\tfrac{1}{\vt'(1)}
\vt\left(\tfrac{z_1}{\gamma_1}\right)
\vt\left(\tfrac{z_2}{\gamma_1}\right)
\delta\left(\tfrac{z_1}{\gamma_1},h\right) 
\delta\left(\tfrac{z_2}{\gamma_1},\tfrac{\mu_2}{\mu_1}\right)
.
\end{equation}
For $n=3$ for example we have 
\[
\wwh_{123}=
\frac{
\vt\left(\tfrac{z_2}{\gamma_1}\right)
\vt\left(\tfrac{z_3}{\gamma_1}\right)
\vt\left(\tfrac{z_3}{\gamma_2}\right)
\vt\left(\tfrac{z_1h}{\gamma_2}\right)
\vt\left(\tfrac{z_1 h\mu_3}{\gamma_1\mu_1}\right)   
\vt\left(\tfrac{z_2 h\mu_3}{\gamma_2\mu_2}\right)
}
{
\vt\left(\tfrac{h\mu_3}{\mu_1}\right)
\vt\left(\tfrac{h\mu_3}{\mu_2}\right)
\vt\left(\tfrac{\gamma_1h}{\gamma_2}\right)
}.
\]
\end{example}

The key properties of weight functions are the R-matrix recursion property, substitution properties, transformation properties, orthogonality relations, as well as their axiomatic characterizations, see details in \cite{RTV}. First we recall some obvious substitution properties, and the R-matrix recursion property. 

\subsection{Substitution properties}
Keeping in mind that fixed point restrictions in geometry are obtained by the substitutions~\eqref{substitute}, for permutations $\om,\sigma$ define    
\[
\ww_{\om,\sigma}=
{\wwh_{\om}}|_{\gamma_k=z_{\sigma(k)}}=
{\ww_{\om}}|_{t^{(k)}_i=z_{\sigma(i)}}.
\]
From the definition of weight functions (or by citing \cite[Lemmas 2.4, 2.5]{RTV}) it follows that 
$\ww_{\om,\sigma}=0$ unless $\sigma\leq \omega$ in the Bruhat order, and
\begin{equation*}
\ww_{\omega,\omega}={\prod_{i<j}\vt(z_{\om(j)}/z_{\om(i)})} \cdot \prod_{i<j,\om(i)>\om(j)} \delta(z_{\omega(j)}/z_{\omega(i)},h).
\end{equation*} 
In particular, we have
\begin{equation}\label{Rbasic}
{\ww_{\id,\sigma}}=\begin{cases}
{\prod_{i<j}\vt(z_{j}/z_{i})} & \text{if } \sigma=\id \\
0 & \text{if } \sigma\not=\id.
\end{cases}
\end{equation}

\subsection{R-matrix recursion} 
In \cite{RTV} (Theorem 2.2 and notation (2.8)) the following identity is proved for weight functions:
\begin{equation}\label{eq:Rmatrix}
s_k^z\ww_{s_k\om}  = 
\begin{cases}
\ww_{\om} \cdot \frac{\delta(\frac{\mu_{\om^{-1}(k)}}{\mu_{\om^{-1}(k+1)}},h) \delta(\frac{\mu_{\om^{-1}(k+1)}}{\mu_{\om^{-1}(k)}},h)}{\delta(\frac{z_k}{z_{k+1}},h)}   
- 
\ww_{s_k\om} \cdot \frac{\delta( \frac{z_{k+1}}{z_k}, \frac{\mu_{\om^{-1}(k)}}{\mu_{\om^{-1}(k+1)}})}{\delta(\frac{z_k}{z_{k+1}},h)}
& \text{if } \ell(s_k \om)>\ell(\om)
\\
\ww_{\om} \cdot \frac{1}{\delta(\frac{z_k}{z_{k+1}},h)}   
- 
\ww_{s_k\om} \cdot \frac{\delta( \frac{z_{k+1}}{z_k}, \frac{\mu_{\om^{-1}(k)}}{\mu_{\om^{-1}(k+1)}})}{\delta(\frac{z_k}{z_{k+1}},h)}
& \text{if } \ell(s_k \om)<\ell(\om),
\end{cases}
\end{equation}
where $s^z_k$ operates by replacing the $z_k$ and $z_{k+1}$ variables. Of course, the same formula holds for $\wwh$--functions (replace $\ww$ with $\wwh$ everywhere in \eqref{eq:Rmatrix}).

\begin{corollary} \label{cor:R-recur}
If $\ell(s_k\omega)=\ell(\omega)+1$ then
\begin{equation*}
\ww_{s_k\om}    
= 
\delta\left( \frac{z_{k+1}}{z_k}, \frac{\mu_{\om^{-1}(k+1)}}{\mu_{\om^{-1}(k)}}\right)
\cdot\ww_{\om}  + \delta\left(\frac{z_k}{z_{k+1}},h\right)\cdot s_k^z\ww_{\om} 
\end{equation*}
and the same holds if $\ww$ is replaced with $\wwh$.
\end{corollary}

\begin{proof} The statement follows from the second line of \eqref{eq:Rmatrix}, after we rename $\omega$ to $s_k \omega$. 
\end{proof}

A key observation is that the recursion in Corollary \ref{cor:R-recur}, together with the initial condition \eqref{Rbasic} completely determine the classes $\ww_{\om,\sigma}$.

\section{Weight functions are representatives of elliptic classes}
For a rank $n$ bundle $\TTT$ with Grothendieck roots $t_k$ we defined its K theoretic Euler class in~\eqref{eq:eK}. 
We recall that the elliptic cohomology is understood as Borel equivariant cohomology $\hat H^*_\T(-;\CC)[[q]]$ with the complex orientation given by the theta function.
The elliptic Euler class is defined by $\eee^{\rm ell}(\TTT)=\prod_{k=1}^n  \theta(\xi_k)$, where $\xi_k$ are the Chern roots, $ch(t_k)=\eee^{\xi_k}$. Below we identify equivariant K theory or equivariant elliptic cohomology classes of $G/B$ with the tuple of their restrictions to the fixed points.  Formally we should apply the Chern character to compare formulas in $$K_\T(pt)[[q]]\simeq \ZZ[z_1^{\pm 1},z_2^{\pm 1},\dots,z_n^{\pm 1}][[q]]$$ with 
$$\hat H^*_\T(pt;\CC)[[q]]\simeq \CC[[x_1,x_2,\dots,x_n,q]]\,.$$ 
Here the variables $z_k$ form the basis of characters of $\T$, while the variables $x_k$ are the weights, they form an integral basis of $\t^*$.
The Chern character is given by the substitution $$z_k\mapsto \eee^{x_k}$$
and it is clearly injective. Therefore from now on we will omit the Chern character in the notation and for example we will write $\vt(z_i/z_j)$ instead of $\theta(x_i-x_j)$.
\medskip

We will be concerned with $\eell(T\!\Fl(n))_{|\sigma}=\eell(T_\sigma\!\Fl(n))=\prod_{i<j} \vt(z_{\sigma(j)}/z_{\sigma(i)})$ for a permutation $\sigma$. Using this Euler class, the recursion we obtained in the last section, \eqref{Rbasic}, reads:
\begin{equation}\label{locini}
\frac{\ww_{\id,\sigma}}{\eell(T_{\sigma}\Fl(n))}=\begin{cases} 1 & \text{if } \sigma=\id \\ 0 & \text{if } \sigma\not=\id, \end{cases} 
\end{equation}
and for $\ell(s_k \om)=\ell(\om)+1$
\begin{equation}\label{eq:Rw-recurence}
\frac{\ww_{s_k\om,\sigma}}{\eell(T_\sigma\!\F(n))}
= 
\delta\left( \tfrac{z_{k+1}}{z_k}, \tfrac{\mu_{\om^{-1}(k+1)}}{\mu_{\om^{-1}(k)}}\right)
\cdot\frac{\ww_{\om,\sigma}}{\eell(T_\sigma\!\F(n))}  + \delta\left(\tfrac{z_k}{z_{k+1}},h\right)\cdot s_k^z\left(\frac{\ww_{\om,s_k\sigma}}{\eell(T_{s_k\sigma}\!\F(n))}\right).
\end{equation}

Now we are ready to state the theorem that weight functions represent the elliptic classes of Schubert varieties.

\begin{theorem}\label{thm:EsameW}
Set $\mu_i=h^{-\lambda_i}$. With this identification
in presentation \eqref{ring1} we have 
\[
\Ek(X_\om,\lambda)=\frac{\eekk(\TFn)}{\eell(\TFn)}\cdot [\ww_\om] \,,
\]
and in presentation \eqref{ring2} we have
\[
\Ek(X_\om,\lambda)= \frac{\eekk(\TFn)}{\eell(\TFn)}\cdot [\wwh_\om]\,.
\]
\end{theorem}

\begin{remark} \rm \label{rm:ell2K}
Continuing Remark \ref{rem:EllH} let us note that if we had set up the elliptic class of varieties not in equivariant K theory but in equivariant elliptic cohomology, then the class would be multiplied by ${\eell(TM)}/{\eekk(TM)}$ (where $M$ is the ambient space). That is, Theorem \ref{thm:EsameW} claims that the functions $\ww_\om$, $\wwh_\om$ represent the elliptic class of Schubert varieties in equivariant elliptic cohomology. 
\end{remark}

\subsection{Proof of Theorem \ref{thm:EsameW}}
Let us fix a notation:

\medskip
\noindent{\bf Convention.} We will skip $\lambda$ in the notation of $E_\sigma(X_\om,\lambda)$ and we treat $E_\sigma(X_\om)$ as a function on $\lambda\in\t^*\simeq \C^n$ expressed by  
the basic functions $\mu_k=\mu_k(\lambda)=h^{-\lambda_k}$. The action of $\om\in W$ on the space of functions generated by $\mu_\bullet$ will be denoted by $s_k^\mu$. We will write $L_{k,\sigma}$ to denote the character of the line bundle $L_k$ at the point $x_\sigma$.
\bigskip

\noindent We need to prove that for all $\om, \sigma$ 
\[
E_\sigma(X_{\om})=\frac{\ww_{\om,\sigma}}{\eell(T_{\sigma}\!\F(n))}.
\]
This will be achieved by showing that the recursive characterization \eqref{locini}, \eqref{eq:Rw-recurence} of the right hand side holds for the left hand side too.
The basic step \eqref{locini} holds for $E_\sigma(X_{\id})$ because of  \eqref{Eini}.

\begin{proposition}\label{pro:ERrecursion}Suppose $G=\GL_n$. If $\ell(s_k\om)=\ell(\om)+1$ then the functions $E_\sigma(X_\om)$ satisfy the recursion
$$E_\sigma(X_{s_k\om})=
\delta\left(\frac{z_{k + 1}}{z_k},\frac{ \mu_{\om^{-1}(k + 1)}}{\mu_{\om^{-1}(k)}}\right) 
\cdot E_\sigma(X_\om) + 
\delta\left(\frac{z_k }{z_{k+1}},h\right)\cdot
s_k^z E_{s_k\sigma}(X_\omega).$$
More generally for an arbitrary reductive group 
$$E_\sigma(X_{s_k\om})=
\delta\left(L_{k,\id},(\om^{-1})^\mu(\nu_k)\right) 
\cdot E_\sigma(X_\om) + 
\delta\left(L^{-1}_{k,\id},h\right)\cdot
s_k^zE_{s_k\sigma}(X_\omega)\,.$$
Here 
$s_k^z$ 
means the action of $s_k$ on $z$-variables and $\omega^\mu$ acts on the $\mu$-variables, $\nu_k=h^{-\alpha_k^\vme}$.\end{proposition}

The reader may find it useful to verify the statement for $n=2$ using the local classes below.   
\begin{center}{\def\arraystretch{2}
\begin{tabular}{|c|l|l|}
\hline
&$\om=\id$&$\om=s_1$\\
\hline
$E_\sigma(X_\om)$ &$E_{12}(X_{12})=1 $  &  $E_{21}(X_{12})=0$ \\
& $E_{12}(X_{21})=\delta(\frac{z_2}{z_1},\frac{\mu_2}{\mu_1})$    &  $E_{21}(X_{12})=\delta(\frac{z_1}{z_2},h)$ \\
\hline
$\ww_{\om,\sigma}$ & $\ww_{12,12} =\vt(\frac{z_2}{z_1})  $                                       &  $\ww_{12,21}=0$ \\
& $\ww_{21,12}=\vt(\frac{z_2}{z_1}) \delta(\frac{z_2}{z_1},\frac{\mu_2}{\mu_1})$ &  $\ww_{21,21}=\vt(\frac{z_1}{z_2}) \delta(\frac{z_1}{z_2},h)$ \\
\hline
\end{tabular}}
\end{center}

\begin{proof}
We prove the proposition by induction with respect to the length of $\om$. We assume that the formula of Proposition \ref{pro:ERrecursion} holds for $\om_1$ with $\ell(\om_1)<\ell(\om)$. 
We introduce the notation for a group element  inverse:
$\bb\om=\om^{-1}$. Note that $\bb{s_k}=s_k$.
Let's assume that  $\omm s_\ell$ is a reduced expression of $\om$ and $\ell(s_k\om)=\ell(\om)+1$. Then  
$\ell(s_k\om_1)=\ell(\om_1)+1$.
By Theorem \ref{th:mainiduction}
$$E_\sigma(X_{s_k\om})=\delta(L_{\ell,\sigma},\nu_\ell)s_\ell^\mu E_\sigma(X_{s_k\om_1 })+\delta(L_{\ell,\sigma},h)s_\ell^\mu E_{\sigma s_\ell}(X_{s_k\om_1 })$$
By inductive assumption this expression is equal to
\begin{multline*}\delta(L_{\ell,\sigma},\nu_\ell)s_\ell^\mu\big(\delta(L_{k,\id},\bb\om_1^\mu(\nu_k))E_\sigma(X_{\om_1})+\delta(L_{k,\id}^{-1},h)s_k^z E_{s_k\sigma}(X_{\om_1})\big)\\
+\delta(L_{\ell,\sigma},h)s_\ell^\mu\big(
\delta(L_{k,\id},\bb\om_1^\mu(\nu_k))E_{\sigma s_\ell}(X_{\om_1})+\delta(L_{k,\id}^{-1},h)s_k^z E_{s_k\sigma s_\ell}(X_{\om_1})\big)\,.\end{multline*}
Note that 
\[\begin{array}{ll}
s_k^z\delta(L_{\ell,\sigma},\nu_k)=\delta(L_{\ell,s_k\sigma},\nu_k)
\,,&
s_k^z\delta(L_{\ell,\sigma},h)=\delta(L_{\ell,s_k\sigma},h)\,,
\\ \\
s_\ell^\mu\delta(L_{k,\sigma},\bb\om_1^\mu(\nu_k))=\delta(L_{k,\sigma},s_l^\mu\bb\om_1^\mu(\nu_k))\,,
&
s_\ell^\mu\delta(L_{k,\id}^{-1},h)=\delta(L_{k,\id}^{-1},h)\,,
\end{array}\]
hence rearranging the expression we obtain
\begin{multline*}
\delta(L_{k,\id},s_\ell^\mu\bb\om_1^\mu(\nu_k))\big(
\delta(L_{\ell,\sigma},\nu_\ell)s_\ell^\mu E_\sigma(X_{\om_1})+
\delta
(L_{\ell,\sigma},h)s_\ell^\mu E_{\sigma s_\ell}(X_{\om_1})\big)
\\
+\delta(L_{k,\id}^{-1},h)
s_k^z\big(\delta(L_{\ell,s_k\sigma},\nu_\ell)s_\ell^\mu  E_{s_k\sigma}(X_{\om_1}))
+
\delta(L_{\ell,s_k\sigma},h)s_\ell^\mu  E_{s_k\sigma s_\ell}(X_{\om_1})\big)
=\end{multline*}

\begin{multline*}=\delta(L_{k,\id},s_\ell^\nu\bb\om_1^\mu(\nu_k))
E_\sigma(X_{\om_1 s_\ell})+
\delta(L_{k,\id}^{-1},h)
s_k^z E_{s_k\sigma}(X_{\om_1 s_\ell})=\\
=\delta(L_{k,\id},\bb\om^\mu(\nu_k))
E_\sigma(X_{\om})+
\delta(L_{k,\id}^{-1},h)
s_k^z E_{s_k\sigma}(X_{\om})\,.
\end{multline*}

\end{proof}

\noindent This completes the proof of Theorem \ref{thm:EsameW}.

\begin{proposition} 
If $\ell(s_k\om)=\ell(\om)-1$, then
$$\delta\big(\frac{z_{k + 1}}{z_k},\frac{ \mu_{\om^{-1}(k + 1)}}{\mu_{\om^{-1}(k)}}\big) 
\cdot E_\sigma(X_\om) + 
\delta\big(\frac{z_k }{z_{k+1}},h\big)\cdot E_{s_k\sigma}(X_\om)=\delta\big(h,\frac{ \mu_{\om^{-1}(k + 1)}}{\mu_{\omega^{-1}(k)}}\big)
\delta\big(h,\frac{ \mu_{\om^{-1}(k )}}{\mu_{\om^{-1}(k+1)}}\big)
E_\sigma(X_{s_kw})\,.$$\end{proposition}
\begin{proof}
For $G=\GL_n$
this relation is a reformulation of the first line of the R-matrix relation \eqref{eq:Rmatrix}.  For general $G$ this statement follows from a direct calculation when $\om_1=s_k \om_2$, $\ell(\om_2)<\ell(\om_1)$, which is exactly the same as the proof of Theorem \ref{th:kappa}. 
\end{proof} 

To obtain Formula \ref{R-intro} given in the Introduction we multiply and divide by the elliptic Euler classes \eqref{eq:ellipticEuler}. Here, comparing with the proof of Formula \ref{BS-intro}, the minus sign does not appear because additionally we have the action of $s_k^z$ compensating the sign.

\section{Transformation properties of $\Ek(X_{\om})$}

Having proved that for $\GL_n$ the elliptic classes of Schubert varieties are represented by weight functions, we can conclude that all proven properties of weight functions hold for those elliptic classes. One key property of weight functions is a strong constraint on their transformation properties. Hence, such a constraint holds for $\Ek(X_\om)$ in the $\GL_n$ case. Motivated by this fact we will prove an analogous theorem on the transformation properties of elliptic classes of Schubert varieties for any reductive group $G$. 
	
In Section \ref{sec:transformations} we recall how to encode transformation properties of theta-functions by quadratic forms and recall the known transformation properties of weight functions. In Section~\ref{sec:axiom} we put that statement in context be recalling a whole set of other properties such that together they characterize weight functions. In Section \ref{sec:trans_general} we generalize the transformation property statement to arbitrary $G$. So the new result is only in Section \ref{sec:trans_general}, the preceding sections are only motivations for that. 

\subsection{Transformation properties of the weight function} \label{sec:transformations}
Consider functions $\C^p\times \HH \to \C$, where $\HH$ is the upper half space, and the variable in $\HH$ is called $\tau$.  
Let $M$ be a $p\times p$ symmetric integer matrix, which we identify with the quadratic form $x\mapsto x^TMx$. 
We  say that the function $f:\C^p \times \HH \to \C$ has transformation property $M$, if 
\begin{align*}
f(x_1,\ldots,x_{j-1},\,x_j+2\pi i,\,x_{j+1},\ldots,x_p)=& (-1)^{M_{jj}} f(x), \\
f(x_1,\ldots,x_{j-1},\,x_j+2\pi i\tau,\,x_{j+1},\ldots,x_p)=&(-1)^{M_{jj}} \eee^{-\text{$\Sigma_k$} M_{jk}x_k-\pi i \tau M_{jj}} f(x).
\end{align*}

For a quadratic form $M$ there one may define a line bundle $\mathcal L(M,0)$ over the $p$th power of the elliptic curve $\C/2\pi i\langle 1,\tau\rangle$ such that the sections of $\mathcal L(M,0)$ are identified with functions with transformation properties $M$, see \cite[Section 6]{RTV}.

Recall from Section \ref{sec:theta} that we set up theta function in two ways, $\vt(\ )$ 
in ``multiplicative variables'', and $\theta(\ )$ 
in ``additive variables''. The transformation property being the quadratic form $M$ is always meant in the additive variables, but naming the quadratic function we use the variable names most convenient for the situation. For example, the function $\vt(x)$ has transformation property $M=(1)$ (or equivalently, the quadratic form $x^2$), because of \eqref{theta_trans} .

Iterating this fact one obtains that for integers $r_i$ the function 
$\vt(\prod_{i=1}^p x_i^{r_i} )$ has transformation property $(\sum_{i=1}^p r_ix_i)^2$. For products of functions the quadratic form of their transformation properties add. Hence, for example, the function $\delta(a,b)$  of \eqref{def:delta} has transformation property $(a+b)^2-a^2-b^2=2ab$. Through careful analysis of the combinatorics of the weight functions defined above (or from \cite[Lemmas 6.3, 6.4]{RTV} by carrying out the necessary convention changes) we obtain

\begin{proposition}
The weight function $\wwh_\om$ has transformation property 
\begin{align}\label{Qpi}
\notag Q(\om)=&\mathop{\sum_{1\leq i,j\leq n-1}}_{\om(i)<\om(j)} 2h(\gamma_j-\gamma_i) + \sum_{i=1}^{n-1}\sum_{j=1}^{\om(i)-1} 2h(z_j-\gamma_i) \\
                    &+ \sum_{i=1}^{n-1} 2(z_{\om(i)}-\gamma_i) (P_{\om,i}h+\mu_n-\mu_i) \\
\notag           &+ \sum_{i=1}^{n-1} \sum_{j=1}^n (z_j-\gamma_i)^2-\mathop{\sum_{1\leq i,j\leq n-1}}_{i<j} (\gamma_i-\gamma_j)^2,
\end{align}
for $P_{\om,i}=1$ if $\om(i)<\om(n)$ and $P_{\om,i}=0$ otherwise. \qed
\end{proposition}

\begin{example} \rm
We have 
\begin{align*}
Q(12)&=2(z_1-\gamma_1)(h+\mu_2-\mu_1)+(z_1-\gamma_1)^2+(z_2-\gamma_1)^2,\\
Q(21)&=2h(z_1-\gamma_1)+2(z_2-\gamma_1)(\mu_2-\mu_1)+(z_1-\gamma_1)^2+(z_2-\gamma_1)^2,
\end{align*}
in accordance with formulas \eqref{WWex}.
\end{example}

\begin{corollary} \label{cor:transfom_recursive}
We have
\begin{align*} 
Q(\om)-Q(\om s_k)=& 2(z_{\om(k)}-z_{\om(k+1)})(\mu_{k+1}-\mu_k), \\
Q(\om)-Q(s_k\om)=& 2(z_{k}-z_{k+1})(\mu_{\om^{-1}(k+1)}-\mu_{\om^{-1}(k)}).
\end{align*}
\end{corollary}
Note that neither line depends on the variables $\gamma_i$.
\begin{proof} 
Straightforward calculation based on formula \eqref{Qpi}, carried out separately in the few cases depending on whether $k<n-1$ or $k=n-1$, whether $\om(k)<\om(n)$ or not. 
Let us note that a more conceptual proof for (only) the second line can be obtained using the R-matrix relation~\eqref{eq:Rmatrix}.
\end{proof}

\subsection{Axiomatic characterization}\label{sec:axiom}
In this section we recall from \cite{RTV} a list of axioms that determine the weight function. One of the axioms is that they must have the transformation properties calculated above. Some other axioms include holomorphicity of some functions. Since in the definition of $\vt(x)$ the square root of $x$ appears, the domain of our functions is a suitable cover of $\C^N$ (or the domain of the induced section is a suitable cover of the product of elliptic curves).



The ``constant'' (not depending on $z$ variables)  
\[
\psi_\om=\psi_\om(\mu,h)=\vt(h)^{n(n-1)(n-2)/3}
\prod_{i<j}\left( \prod_{\om(i)<\om(j)} \vt\left( h\mu_j/\mu_i\right) \cdot \prod_{\om(i)>\om(j)} \vt(h) \vt\left(\mu_j/\mu_i\right) \right).
\]
plays a role below.


\begin{theorem}[\cite{RTV} Theorem 7.3, see also \cite{FRV} Theorem A.1] \label{thm:axioms} \ \\
 (I) The functions $\ww_{\om,\sigma}$ satisfy the properties:
\begin{itemize}
  \item[(1.1)] (holomorphicity) We have 
  \[ 
  \ww_{\om,\sigma}=\frac{1}{\psi_\om} \cdot \text{holomorphic function.}
  \]
  \item[(1.2)] (GKM relations) We have 
  \[
  \ww_{\om,\sigma s_k}|_{z_{\sigma(k)}=z_{\sigma(k+1)}}=\ww_{\om,\sigma}|_{z_{\sigma(k)}=z_{\sigma(k+1)}}.
  \]
  \item[(1.3)] (transformations) The transformation properties of $\ww_{\om,\sigma}$ are described by the quadratic form $Q(\om)_{|\gamma_i=z_{\sigma(i)}}$.
  \item[(2)]\label{item:smooth} (normalization)
  \[
  \ww_{\om,\om}=
  \prod_{i<j} \vt(z_{\om(j)}/z_{\om(i)})
  \mathop{\prod_{i<j}}_{\om(j)<\om(i)} \delta(z_{\om(j)}/z_{\om(i)},h).
  \]
  \item[(3.1)] (triangularity) if $\sigma\not\preccurlyeq \om$ in the Bruhat order then $\ww_{\om,\sigma}=0$.
  \item[(3.2)] (support)  if $\sigma\preccurlyeq \om$ in the Bruhat order then $\ww_{\om,\sigma}$ is of the form
  \[\frac{1}{\psi_\om} \cdot \mathop{\prod_{i<j}}_{\sigma(i)>\sigma(j)} \vt(z_{\sigma(j)}h/z_{\sigma(i)}) \cdot \text{holomorphic function}
  \] 
\end{itemize}
(II) These properties uniquely determine the functions $\ww_{\om,\sigma}$.  \qed
\end{theorem}

\noindent The axiom (1.3) may be replaced by the inductive property:
\begin{itemize}
\item[(1.3')] If $\om=\om's_k$ with $\ell(\om)=\ell(\om')+1$  
then the difference of the quadratic forms describing the transformation properties of 
$\ww_{\om',\sigma}$ 
and that of 
$\ww_{\om,\sigma}$  
is equal to 
$$2(z_{\om(k+1)}-z_{\om(k)})(\mu_{k+1}-\mu_k).$$
\end{itemize}

\subsection{Transformation properties for general $G$} \label{sec:trans_general}
A key axiom for the weight functions is (1.3'). Through Theorem \ref{thm:EsameW} it implies that the difference of the quadratic forms of $E_\sigma(X_{\om'},\lambda)$ and $E_\sigma(X_{\om},\lambda)$ is also $2(z_{\om(k+1)}-z_{\om(k)})(\mu_{k+1}-\mu_k)$.
Below, in Theorem \ref{simple_induction} we will prove that the generalization of this surprising property holds for general $G$. We will also see that this general argument (using the language of a general Coxeter groups) fits better the transformation properties than the combinatorics of weight functions. 

\smallskip

Let $\alpha$ and $\beta\in \t^*$ be two roots. The reflection of $\alpha$ about $\beta$ will be denoted by $\alpha^\beta=s_\beta(\alpha)$.
Define the functional $\mu_\alpha$ acting on $\t^*$
$$\mu_\alpha\in (\t^*)^*,\qquad
\mu_\alpha(\beta)=-\langle \beta, \alpha^\vme \rangle.$$
Note that we artificially introduce the minus sign to agree  with the previous conventions (the definition of $\LL{\lambda}$ and $\mu_k=h^{-\lambda_k}$ in the weight function).
The reflection $s_\alpha\in W$ acting on polynomial functions on $\t^*$ is denoted by $s_\alpha^\mu$
$$s^\mu_\alpha(f)=f\circ s_\alpha.$$
Define the generalized divided difference operation
$$d_\alpha(f)(x)=\frac{f(x)-f(s_\alpha(x))}{\mu_\alpha(x)}.$$
It satisfies the properties
\[
s^\mu_\alpha=s^\mu_{-\alpha},
\qquad 
d_\alpha=-d_{-\alpha}, 
\qquad
d_{\beta} \circ s^\mu_\alpha=s^\mu_\alpha \circ d_{\beta^\alpha},
\qquad
d_\beta(\mu_\alpha)=\langle \alpha,\beta^\vme  \rangle.
\]
In particular, we have $d_{\alpha}\mu_\alpha=\langle\alpha, \alpha^\vme  \rangle=2$.

Consider the vector space $\t\times\t^*\times \CC$. For a root $\alpha\subset \t^*$ the linear functional $z_\alpha\in (\t\times\t^*\times\CC)^*$ depends on the first coordinate by $\alpha$.
The functional $\mu_\alpha$ depends on the second coordinate, while $h$ depends on the third coordinate. For $\sigma\in W$ and $\alpha\in \t^*$ by $\sigma(\alpha)$ we understand the usual action of $\sigma$ on $t^*$.
\medskip

As usual, we keep fixed the positive roots and the simple roots. 
To each pair  ${\om,\sigma}\in W$ such that $\sigma\preccurlyeq \om$  we associate a quadratic form $M(\om,\sigma)$ such that $M(\id,\id)=0$ 
and inductively: If $\om=\om's_\alpha$ with $\ell(\om)=\ell(\om')+1$, then
\begin{equation}\label{def-qinduction}M(\om,\sigma)=\begin{cases}
s^\mu_\alpha(M(\om',\sigma))+\mu_\alpha\, z_{\sigma(\alpha)}&\text{ if }\sigma\preccurlyeq \om'\\
s^\mu_\alpha(M(\om',\sigma s_\alpha))-h\,z_{\sigma(\alpha)}&\text{ if }\sigma s_\alpha\preccurlyeq \om'.\end{cases}\end{equation}
If $\sigma\preccurlyeq \om'$ and $\sigma s_\alpha \preccurlyeq \om'$, then the cases of the definitions give the same quadratic form (this follows from the proof of Proposition \ref{prop88} below). Also, $\sigma\preccurlyeq \om$ implies that one of the above conditions holds.
\medskip

\begin{example}\rm Let $G=\GL_3$. We illustrate two ways of computing $M(s_1s_2s_1,id)$:
$$\begin{array}{cc}
 M(\id,\id)= & \phantom{aa}0 \\
&^{s_1}\downarrow\\
 M(s_1,\id)= & \boxed{\left(z_2-z_1\right) \left(\mu _2-\mu _1\right)}
   \\
&^{s_2}\downarrow\\
 M(s_1s_2,\id)= & \left(z_2-z_1\right) \left(\mu _3-\mu
   _1\right)+\boxed{\left(z_3-z_2\right) \left(\mu _3-\mu
   _2\right)} \\
&^{s_1}\downarrow\\
 M(s_1s_2s_1,id)= &\left(z_2-z_1\right) \left(\mu _3-\mu
   _2\right)+\left(z_3-z_2\right) \left(\mu _3-\mu
   _1\right)+\boxed{ \left(z_2-z_1\right) \left(\mu _2-\mu
   _1\right)}. \\
\end{array}
$$
But also
$$\begin{array}{cc}
M(\id,\id)= & 0 \\&^{s_1}\downarrow\\
M(s_1,s_1)= & \boxed{ h\left(z_1-z_2\right) }\\&^{s_2}\downarrow\\
M(s_1s_2,s_1)= & h \left(z_1-z_2\right)+\boxed{\left(z_3-z_1\right)
   \left(\mu _3-\mu _2\right)} \\&^{s_1}\downarrow\\
M(s_1s_2s_1,\id)= & h \left(z_1-z_2\right)+\left(z_3-z_1\right) \left(\mu
   _3-\mu _1\right)+
   \boxed{ h\left(z_2-z_1\right)}. \\
\end{array}
$$
In both cases we obtain $M(s_1s_2s_1,\id)=\left(z_3-z_1\right) \left(\mu
   _3-\mu _1\right)$.
One can check that presenting this permutation as $s_2s_1s_2$ and performing analogous computations we obtain the same result.
\end{example}

The inductive procedure of constructing the elliptic classes can be translated to a description of the associated quadratic form.
\begin{proposition}\label{pr:ellform} The quadratic form $2M(\om,\sigma)$  describes the transformation properties of $E_\sigma(X_\om,\lambda)$.\end{proposition}
\begin{proof} If $\om=\id$ then $$E_\sigma(X_{\id},\lambda)=\begin{cases}1&\text{ if } \id=\sigma\,\\ 0&\text{ if } \id\neq\sigma,\end{cases}$$
hence the associated form at $(\id,\id)$ is 0. By Theorem \ref{th:mainiduction}, when passing from $\om'$ to $\om$ the elliptic class changes by the factor $\delta^{in}_k=\delta(L_k,h)$, which is at the point $\sigma$ has the transformation properties $2z_{-\sigma(\alpha_k)} h=-2z_{\sigma(\alpha_k)}h$ (since $L_k=\LL{\alpha_k}$) or by the factor $\delta^{bd}_k=\delta(L_k,h^{\langle\lambda,\alpha_k^\vme\rangle})$  which at the point $\sigma$  has the transformation properties $2z_{-\alpha_k}\mu_{-\alpha_k}=2z_{\sigma(\alpha_k)}\mu_{\alpha_k}$.
\end{proof}
\begin{proposition}\label{pr:center}The quadratic form at the smooth point of the cell is equal to 
$$M(\om,\om)=h\sum_{\alpha\,\in\,\Sigma_+\cap\,\om(\Sigma_-)}  z_{\alpha},$$
where $\Sigma_\pm$ denotes the set of positive/negative roots. The roots appearing in the summation are the tangent weights of $X_\om$ at $x_\om$.\end{proposition}
\begin{proof}At the smooth point $x_\om$  the localized elliptic class $E_\om(X_\om,\lambda)$is given by the product of $\delta$ functions and the transformation matrix for $\delta(x,h)$ is equal to $2xh$ by \eqref{EllResLoc1}.\end{proof}

\begin{proposition} \label{prop88}
Suppose $\sigma\preccurlyeq \om$, then
$$d_{\beta}(M(\om,\sigma))=z_{\sigma(\beta)}-z_{\om(\beta)}\,.$$
\end{proposition}

\begin{proof} Obviously, the statement holds of $\om=\id$. 

Consider the first case of the inductive definition. We have
\begin{align*}
d_\beta(M(\om,\sigma)) & =d_\beta(s^\mu_\alpha(M_{\om',\sigma})+\mu_\alpha\, z_{\sigma(\alpha)})\\
&= s^\mu_\alpha d_{\beta^\alpha}(M(\om',\sigma))+\langle \alpha,\beta^\vme\rangle z_{\sigma(\alpha)}\\
& = s^\mu_\alpha(z_{\sigma(\beta^\alpha)}-z_{\om'(\beta^\alpha)})+\langle \alpha,\beta^\vme\rangle z_{\sigma(\alpha)}\\
& =z_{\sigma s_\alpha(\beta)}-z_{\om(\beta)}+\langle\alpha, \beta^\vme\rangle z_{\sigma(\alpha)}.
\end{align*}
Since 
\[
s_\alpha(\beta)=\beta-\langle \alpha,\beta^\vme\rangle \alpha, \ \ 
s_\alpha(\beta)+\langle \alpha,\beta^\vme\rangle \alpha=\beta,\ \ 
z_{\sigma s_\alpha(\beta)}+\langle \alpha,\beta^\vme\rangle z_{\sigma(\alpha)}=z_{\sigma(\beta)}
\]
the conclusion follows.
\medskip

Consider the second case of the inductive definition. We have
\begin{align*}
d_\beta(M(\om,\sigma))& =d_\beta(s^\mu_\alpha(M_{\om',\sigma s_\alpha})-h\, z_{\sigma(\alpha)})\\
&=s^\mu_\alpha d_{\beta^\alpha}(M(\om',\sigma s_\alpha))\\
&=s^\mu_\alpha(z_{\sigma s_\alpha(\beta^\alpha)}-z_{\om'(\beta^\alpha)})\\
& =s^\mu_\alpha(z_{\sigma (\beta)}-z_{\om(\beta)})\\
& =z_{\sigma (\beta)}-z_{\om(\beta)}.
\end{align*} 
\end{proof}

Note that both cases of the inductive definition \eqref{def-qinduction} are linear in $\mu_\alpha$ variables.
If both cases are applicable in the above proof, then we get the same result of the differential $d_\beta$ for any root~$\beta$:
$$d_\beta(s^\mu_\alpha(M(\om',\sigma))+\mu_\alpha\, z_{\sigma(\alpha)})=
d_\beta(
s^\mu_\alpha(M(\om',\sigma s_\alpha))-h\,z_{\sigma(\alpha)})\,.$$
It follows that the quadratic forms are equal:
$$s^\mu_\alpha(M(\om',\sigma))+\mu_\alpha\, z_{\sigma(\alpha)}=
s^\mu_\alpha(M(\om',\sigma s_\alpha))-h\,z_{\sigma(\alpha)}\,.$$ 
Therefore the two cases of the formula \eqref{def-qinduction} do not create a contradiction.

\begin{theorem}\label{simple_induction}
 Let $\alpha$ be a simple root. Suppose $\sigma\preccurlyeq\om\preccurlyeq\om s_\alpha$.
Then
$$M(\om s_\alpha,\sigma)=M(\om,\sigma)+\mu_\alpha z_{\om(\alpha)}.$$
\end{theorem}

\begin{proof}
From the  inductive definition of $M$ (case 1) and Proposition \ref{prop88} it follows: 
\begin{align*}
M(\om,\sigma)-M(\om s_\alpha,\sigma) &=M(\om,\sigma)-\big(s^\mu_\alpha(M(\om,\sigma))+\mu_\alpha z_{\sigma(\alpha)}\big)\\
&=\mu_\alpha d_\alpha (M(\om,\sigma))-\mu_\alpha z_{\sigma(\alpha)}\\
&=\mu_\alpha \big(z_{\sigma(\alpha)}-z_{\om(\alpha)} -z_{\sigma(\alpha)}\big)\\
&=-\mu_\alpha z_{\om(\alpha)}.
\end{align*}
\end{proof}
Theorem \ref{simple_induction} is an extension of the property (1.3') of the axiomatic characterization of the weight function to the case of general $G$.
Of course, this inductive step 
together with the diagonal data determines all transformation properties  $M(\om,\sigma)$:

\begin{proposition}\label{pr:comparing}Suppose two families of quadratic forms $M_1(\om,\sigma)$ and $M_2(\om,\sigma)$ are defined for $\sigma\preccurlyeq\om$ and satisfy the  formula of Theorem \ref{simple_induction}. Moreover suppose $M_1(\sigma,\sigma)=M_2(\sigma,\sigma)$ for all $\sigma\in W$, then $M_1(\om,\sigma)=M_2(\om,\sigma)$ for all $\sigma\preccurlyeq\om$.
 \end{proposition}
\begin{proof}We show equality of forms inductively, keeping the second variable of $M_i(\om,\sigma)$ fixed. We can only change $\om$ by an elementary reflection $s_i$. The starting point for the induction is $M_1(\sigma,\sigma)=M_2(\sigma,\sigma)$. Increasing the length of $\om$ by 1 we can arrive to $M_1(\om_0,\sigma)=M_2(\om_0,\sigma)$. Now decreasing the length $\om$ for we can go down to any $\om$ satisfying $\sigma\prec\om\prec\om_0$.
\end{proof}

\section{Weight function vs lexicographically smallest reduced word} \label{eg:FL3}

Let us revisit Example \ref{GL3recu}, and study the underlying geometry and its relation with the weight functions.

The class $E_{123}(X_{321})$ is calculated in Example \ref{GL3recu} in two ways, one corresponding to the reduced word $s_1s_2s_1$, the other corresponding to the reduced word $s_2s_1s_2$ of the permutation $321$. The two obtained expressions are, respectively,
\begin{equation}\label{s1s2s1}
E_{\id}(X_{s_1s_2s_1})=
\delta\left(\frac{z_2}{z_1},\frac{\mu_3}{\mu_2}\right)    
\delta\left(\frac{z_3}{z_2},\frac{\mu_3}{\mu_1}\right)
\delta\left(\frac{z_2}{z_1},\frac{\mu_2}{\mu_1}\right) 
+
\delta\left(\frac{z_1}{z_2},h\right)
\delta\left(\frac{z_3}{z_1},\frac{\mu_3}{\mu_1}\right)
\delta\left(\frac{z_2}{z_1},h\right),
\end{equation}
\begin{equation}\label{s2s1s2}
E_{\id}(X_{s_2s_1s_2})=
\delta\left(\frac{z_3}{z_2},\frac{\mu_2}{\mu_1}\right)
\delta\left(\frac{z_2}{z_1},\frac{\mu_3}{\mu_1}\right)
\delta\left(\frac{z_3}{z_2},\frac{\mu_3}{\mu_2}\right) 
+
\delta\left(\frac{z_2}{z_3},h\right)
\delta\left(\frac{z_3}{z_1},\frac{\mu_3}{\mu_1}\right)
\delta\left(\frac{z_3}{z_2},h\right). 
\end{equation}
As we mentioned, the equality of these two expressions follows from general theory of Borisov and Libgober, or can be shown to be equivalent to the {\em four term identity} \cite[eq.~(2.7)]{RTV}.
\medskip

To explore the underlying geometry, let $u_{21}$, $u_{31}$ and $u_{32}$ be the coordinates in the standard affine neighborhood of the fixed point $x_{123}$, that is, $u_{ij}$ are the entries of the lower-triangular $3\times 3$ matrices. The character of the coordinate $u_{ij}$ is equal to $z_i/z_j$. The open cell $X^\circ_{s_1s_2s_1}$ intersected with this affine neighborhood is the complement of the sum of divisors
$$
\{u_{31}-u_{21}u_{s_2}= 0\}\quad \text{and}\quad 
\{ u_{31}= 0\}\,.$$
The intersections of the divisors is singular. The resolution corresponding to the word $s_1s_2s_1$ coincides with the blow-up of the $u_{21}$ axis. The fiber above $0$ contains two fixed points which give the two contributions in \eqref{s1s2s1}. Analogously, the expression obtained by blowing up the axis $u_{23}$ is \eqref{s2s1s2}.

\smallskip

The reader is invited to verify that the (long) calculation for   
\[
\frac{\ww_{321,123}}{\eell(T_{\id}\!\Fl(3))}=
\frac{\ww_{321,123}}{\vt({z_2}/{z_1})\vt({z_3}/{z_1})\vt({z_3}/{z_2})}
\]
results exactly the expression \eqref{s1s2s1}. This and many other examples calculated by us suggest the following conjecture: the weight function formula for $\ww_{\om,\sigma}/{\eell(T_{\sigma}\!\Fl(n))}$ coincides (without using any $\vt$-function identities) with the expression obtained for $E_\sigma(X_\om)$ using the {\em lexicographically smallest} reduced word  for $\om$.

\section{Action of $C$-operations on weight functions.}
We still consider the case of $G=\GL_n$.
Let $\hC_k$ be a family of operators on the space of meromorphic functions no $\t\times\t^*\times\C$ indexed by the simple roots:
\begin{equation}\label{eq:c-hat}\hC_k(f)(z,\gamma,\lambda,h)=\delta(L^\gamma_k,\nu_k)\cdot f(z,\gamma,s_k(\lambda),h)\,+\,\delta(L^\gamma_k,h)\cdot f(s_k(z),\gamma,s_k(\lambda),h)\,.\end{equation}
Here $$L_k^\gamma=\frac{\gamma_{k+1}}{\gamma_k}$$ denotes the character of the relative tangent bundle $G/B\to G/P_k$ (equal to $\eee^{-\alpha_k}$ at the point $x_{\id}$), but living in the $\gamma$-copy of variables.
We also recall that $$\nu_k=\frac{\mu_{k+1}}{\mu_k}=h^{\langle-,\alpha^\vme_k\rangle}\,.$$
The operators are constructed in such a way that they descend to the operators $C_k$ acting on the K theory of $\F(n)$. One may think about them as  acting on $Frac(K_{\T\times\T}(\Hom(\C^n,\C^n))\otimes \RR(\T^*\times \C^*))[[q]]$.

\begin{theorem}The operators $\hC_k$ satisfy \begin{itemize}
\item braid relations,
\item $\hC_k^2=\kappa(\alpha_k)$.
\end{itemize}
\end{theorem}
\begin{proof}This is straightforward checking and a repetition of the proof of Theorem \ref{th:kappa}.\end{proof}
Let us set ${\mathfrak w}_{id}=\wwh_{id}/\eell_\gamma$, where $\eell_\gamma$ is the elliptic Euler class written in $\gamma$ variables.
The operators $\hC_k$ recursively define the functions ${\mathfrak w}_\om$ living in the same space as the weight function $\wwh(\om)$. They have the restrictions to the fixed points of $\F(n)$ equal to the restrictions of  $\wwh_\om$ divided by the elliptic Euler class. Nevertheless one can check that they are essentially different from the weights function. The difference lies in the ideal defining K theory of $\F(n)$.
\begin{example}\rm Let $G=\GL_2$.
Setting 

$${\mathfrak w}_{id}=\frac{\wwh_{id}}{\eell_\gamma}=
\frac{\vt \left(\frac{z_1}{\gamma _1}\right) \vt
   \left(\frac{z_2}{\gamma _1}\right) \delta
   \left(\frac{z_1}{\gamma _1},\frac{h \mu _2}{\mu
   _1}\right)}{\vt \left(\frac{\gamma _2}{\gamma
   _1}\right)}=\frac{\vt \left(\frac{z_2}{\gamma _1}\right) \vt
   \left(\frac{h \mu _2 z_1}{\gamma _1 \mu
   _1}\right)}{\vt \left(\frac{\gamma _2}{\gamma
   _1}\right) \vt \left(\frac{h \mu _2}{\mu
   _1}\right)}$$
we obtain
$${\mathfrak w}_{s_1}=\frac{\vt \left(\frac{\gamma _2 \mu _2}{\gamma
   _1 \mu _1}\right) \vt \left(\frac{z_2}{\gamma
   _1}\right) \vt \left(\frac{h \mu _1 z_1}{\gamma _1
   \mu _2}\right)}{\vt \left(\frac{\gamma _2}{\gamma
   _1}\right){\!}^2\, \vt \left(\frac{\mu _2}{\mu
   _1}\right) \vt \left(\frac{h \mu _1}{\mu
   _2}\right)}-\frac{\vt \left(\frac{\gamma _2
   h}{\gamma _1}\right) \vt \left(\frac{z_2}{\gamma
   _2}\right) \vt \left(\frac{h \mu _1 z_1}{\gamma _2
   \mu _2}\right)}{\vt \left(\frac{\gamma _2}{\gamma
   _1}\right){\!}^2\, \vt (h)\, \vt \left(\frac{h \mu
   _1}{\mu _2}\right)}
$$

Note that the operations $\hC_i$ (as well as $C_i$) do not preserve the initial transformation form. The  summands of \eqref{eq:c-hat} might have different transformation properties. 
The equality holds in the quotient ring $K_T(\GL_2/B)$.	

\end{example}


\section{A tale of two recursions for weight functions}\label{sec:tale}

\subsection{Bott-Samelson recursion for weight functions}
The main achievement in Sections~\ref{sec:EllWeight}--\ref{eg:FL3} was the identification of the geometrically defined $\Ek(X_\om)$ classes with the weight functions whose origin is in representation theory. The way our identification went was through recursions. The elliptic classes satisfied the Bott-Samelson recursion of Theorem \ref{th:mainiduction}, and the weight functions satisfied the R-matrix recursion of \eqref{eq:Rmatrix}. In Proposition \ref{pro:ERrecursion} we showed that the two recursions are consistent, and hence both recursions hold for both objects.

One important consequence is that (the fixed point restrictions of) weight functions satisfy the Bott-Samelson recursion, as follows:
\begin{theorem} We have
\[
[\wwh_{\om s_k}] =
\begin{cases}
s_k^\mu [\wwh_{\om}] \cdot \delta( \frac{\gamma_{k+1}}{\gamma_k},\frac{\mu_{k+1}}{\mu_k})
-
s_k^\mu s_k^\gamma[\wwh_{\om}] \cdot \delta(\frac{\gamma_{k+1}}{\gamma_k},h)
& \text{if } \ell(\om s_k)>\ell(\om) \\
\\
s_k^\mu[\wwh_{\om}] \cdot 
\frac{\delta( \frac{\gamma_{k+1}}{\gamma_k},\frac{\mu_{k+1}}{\mu_k})}
{\delta(\frac{\mu_k}{\mu_{k+1}},h)\delta(\frac{\mu_{k+1}}{\mu_k},h)}
-
s_k^\mu s_k^\gamma[\wwh_{\om}] \cdot \frac{\delta(\frac{\gamma_{k+1}}{\gamma_k},h)}{\delta(\frac{\mu_k}{\mu_{k+1}},h)\delta(\frac{\mu_{k+1}}{\mu_k},h)}
& \text{if } \ell(\om s_k)<\ell(\om), 
\end{cases}
\]
or equivalently, using the normalization 
\begin{equation}\label{defmodd}
\wwh'_{\om}=\wwh_{\om} \cdot \frac{1}{\prod_{i<j,\om(i)>\om(j)} \delta(\frac{\mu_i}{\mu_j},h)}.
\end{equation}
we have the unified
\begin{equation}\label{uniBSW}
[\wwh'_{\om s_k}] =
s_k^\mu[\wwh'_{\om}] \cdot 
\frac{\delta( \frac{\gamma_{k+1}}{\gamma_k},\frac{\mu_{k+1}}{\mu_k})}
{\delta(\frac{\mu_k}{\mu_{k+1}},h)}
-
s_k^\mu s_k^\gamma[\wwh'_{\om}] \cdot \frac{\delta(\frac{\gamma_{k+1}}{\gamma_k},h)}{\delta(\frac{\mu_k}{\mu_{k+1}},h)}.
\end{equation} \qed
\end{theorem}
This new property of weight functions plays an important role in a followup paper \cite{Sm2} in connection with elliptic stable envelopes. It is worth pointing out that the normalization \eqref{defmodd} also makes the R-matrix property of \eqref{eq:Rmatrix} unified:
\begin{equation}\label{uniRW}
\wwh'_{s_k \om} =
\wwh'_{\om} \cdot 
\frac{\delta( \frac{\mu_{\omega^{-1}(k+1)}}{\mu_{\omega^{-1}(k)}},\frac{z_{k+1}}{z_k})}
{\delta( \frac{\mu_{\omega^{-1}(k)}}{\mu_{\omega^{-1}(k+1)}},h)}
+
s_k^z{\wwh'_{\om}} \cdot \frac{\delta(\frac{z_k}{z_{k+1}},h)}
{\delta( \frac{\mu_{\omega^{-1}(k)}}{\mu_{\omega^{-1}(k+1)}},h)}.
\end{equation}
There is, however, an essential difference between \eqref{uniBSW} and \eqref{uniRW}: the latter holds for the weight functions themselves, while the former only holds for the cosets $[\wwh']$ of $\wwh'$ functions (i.e. after the restriction). Already for $n=2$, $\om=\id$, $k=1$ the two sides of \eqref{uniBSW} only hold for the fixed point restrictions, not for the $\wwh'$ functions themselves.

In essence, the remarkable geometric object $\Ek(X_\om)$ satisfies two different recursions: the Bott-Samelson recursion and the R-matrix recursion. The weight functions are the lifts of $\Ek(X_\om)$ classes satisfying only one of these recursions.

\subsection{Two recursions for the local elliptic classes}
Although we already stated and proved that both Bott-Samelson and R-matrix recursions hold for the elliptic classes, let us rephrase these statements in a convenient normalization.
\begin{itemize}
\item Let $\zeta_k\in \Hom(\T,\C^*)$ be the inverse of a root written multiplicatively, e.g. $\zeta_k=\frac{z_{k+1}}{z_k}$ for $G=\GL_n$   
\item Let $\nu_k\in \Hom(\C^*,\T)$ be the inverse of a coroot written multiplicatively, e.g. $\nu_k=\frac{\mu_{k+1}}{\mu_k}$ for $G=\GL_n$.
\end{itemize}
For $\nu\in \Hom(\C^*,\T)$ define 
\[
\kappa(\nu)=\vt\big(h,\nu\big)\vt\big(h,\nu^{-1}\big).
\]

The Bott-Samelson recursion for local elliptic classes is:
\begin{multline*}
\delta\left(\sigma^z(\zeta_k),\nu_k\right) 
\cdot s^\mu_kE_\sigma(X_\om) + 
\delta\left(\sigma^z(\zeta_k),h\right)\cdot
s^\mu_kE_{\sigma s_k}(X_\omega)=\\
=\begin{cases}E_\sigma(X_{\om.s_k})& \text{if } \ell(\om s_k)=\ell(\om)+1\\ \\
\kappa\left(\nu_k\right)
E_\sigma(X_{\om s_k})&\text{if }\ell(\om s_k)=\ell(\om)-1,\end{cases}
\end{multline*}
and the R-matrix recursion for local elliptic classes is
\begin{multline*}
\delta\left(\zeta_k,\bb\om^\mu(\nu_k)\right) 
\cdot E_\sigma(X_\om) + 
s_k^z\left(\delta\left(\zeta_k,h\right)\cdot
E_{s_k\sigma}(X_\omega)\right)=\\
=\begin{cases}E_\sigma(X_{s_k\om})& \text{if } \ell(s_k\om)=\ell(\om)+1\\ \\
\kappa\left(\bb\om^\mu(\nu_k)\right)
E_\sigma(X_{s_k\om})&\text{if }\ell(s_k\om)=\ell(\om)-1.\end{cases}
\end{multline*}

\noindent The $G=\GL_n$ special case of the above two formulas is 
\begin{multline*}
\delta\left(\frac{z_{\sigma(k + 1)}}{z_{\sigma(k)}},\frac{ \mu_{(k + 1)}}{\mu_{(k)}}\right) 
\cdot s^\mu_kE_\sigma(X_\om) + 
\delta\left(\frac{z_{\sigma(k+1)} }{z_{\sigma(k)}},h\right)\cdot
s^\mu_kE_{\sigma s_k}(X_\omega)=\\
=\begin{cases}E_\sigma(X_{\om.s_k})& \text{if } \ell(\om s_k)=\ell(\om)+1\\ \\
\delta\left(h,\frac{ \mu_{(k + 1)}}{\mu_{(k)}}\right)
\delta\left(h,\frac{ \mu_{(k )}}{\mu_{(k+1)}}\right)
E_\sigma(X_{\om s_k})&\text{if }\ell(\om s_k)=\ell(\om)-1,
\end{cases}\end{multline*}
\begin{multline*}
\delta\left(\frac{z_{k + 1}}{z_k},\frac{ \mu_{\om^{-1}(k + 1)}}{\mu_{\om^{-1}(k)}}\right) 
\cdot E_\sigma(X_\om) + 
\delta\left(\frac{z_k }{z_{k+1}},h\right)\cdot
s_k^zE_{s_k\sigma}(X_\omega)=\\
=\begin{cases}E_\sigma(X_{s_k\om})& \text{if } \ell(s_k\om)=\ell(\om)+1\\ \\
\delta\left(h,\frac{ \mu_{\om^{-1}(k + 1)}}{\mu_{\omega^{-1}(k)}}\right)
\delta\left(h,\frac{ \mu_{\om^{-1}(k )}}{\mu_{\om^{-1}(k+1)}}\right)
E_\sigma(X_{s_k\om})&\text{if }\ell(s_k\om)=\ell(\om)-1.
\end{cases}\end{multline*}

\section{Tables}

\subsection{$G=\GL_3$}
Below we give the full table of localized elliptic classes $E_\sigma(X_\om)$:

$$
\def\arraystretch{2}
\begin{array}{|c|c|c|c|c|c|}\hline
_\sigma\!\!\!{\text{\large$\diagdown$}}\!\!^\om   & ~~\id~~ & s_1 & s_2 & s_1s_2 & s_2s_1 \\
\hline\hline
 \id & 1 & \delt \frac{z_2}{z_1},\frac{\mu _2}{\mu _1}\big) & \delt \frac{z_3}{z_2},\frac{\mu _3}{\mu _2}\big) &
   \delt \frac{z_2}{z_1},\frac{\mu _3}{\mu _1}\big) \delt \frac{z_3}{z_2},\frac{\mu _3}{\mu _2}\big) & \delt
   \frac{z_2}{z_1},\frac{\mu _2}{\mu _1}\big) \delt \frac{z_3}{z_2},\frac{\mu _3}{\mu _1}\big) \\
\hline
 s_1 & 0 & \delt \frac{z_1}{z_2},h\big) & 0 & \delt \frac{z_1}{z_2},h\big) \delt
   \frac{z_3}{z_1},\frac{\mu _3}{\mu _2}\big) & \delt \frac{z_1}{z_2},h\big) \delt
   \frac{z_3}{z_2},\frac{\mu _3}{\mu _1}\big) \\
\hline
 s_2 & 0 & 0 & \delt \frac{z_2}{z_3},h\big) & \delt \frac{z_2}{z_3},h\big) \delt
   \frac{z_2}{z_1},\frac{\mu _3}{\mu _1}\big) & \delt \frac{z_2}{z_3},h\big) \delt
   \frac{z_3}{z_1},\frac{\mu _2}{\mu _1}\big) \\
\hline
 s_1s_2 & 0 & 0 & 0 & \delt \frac{z_1}{z_2},h\big) \delt \frac{z_1}{z_3},h\big) & 0 \\\hline
 s_2s_1 & 0 & 0 & 0 & 0 & \delt \frac{z_1}{z_3},h\big) \delt \frac{z_2}{z_3},h\big) \\\hline
 s_1s_2s_1 & 0 & 0 & 0 & 0 & 0 \\\hline
\end{array}
$$

$$\def\arraystretch{2}  
\begin{array}{|c|c|}
\hline
_\sigma\!\!\!{\text{\large$\diagdown$}}\!\!^\om& s_1s_2s_1=s_2s_1s_2
\\
\hline
\hline
 \id & \delt \frac{z_1}{z_2},h\big) \delt \frac{z_2}{z_1},h\big) \delt \frac{z_3}{z_1},\frac{\mu _3}{\mu
   _1}\big)+\delt \frac{z_2}{z_1},\frac{\mu _2}{\mu _1}\big) \delt \frac{z_2}{z_1},\frac{\mu _3}{\mu
   _2}\big) \delt \frac{z_3}{z_2},\frac{\mu _3}{\mu _1}\big) \\
\hline
 s_1 & \delt \frac{z_1}{z_2},h\big) \delt \frac{z_1}{z_2},\frac{\mu _2}{\mu _1}\big) \delt
   \frac{z_3}{z_1},\frac{\mu _3}{\mu _1}\big)+\delt \frac{z_1}{z_2},h\big) \delt
   \frac{z_2}{z_1},\frac{\mu _3}{\mu _2}\big) \delt \frac{z_3}{z_2},\frac{\mu _3}{\mu _1}\big) \\
\hline
 s_2 & \delt \frac{z_2}{z_1},\frac{\mu _3}{\mu _2}\big) \delt \frac{z_2}{z_3},h\big) \delt
   \frac{z_3}{z_1},\frac{\mu _2}{\mu _1}\big) \\
\hline
 s_1s_2 & \delt \frac{z_1}{z_2},h\big) \delt \frac{z_1}{z_3},h\big) \delt \frac{z_3}{z_2},\frac{\mu
   _2}{\mu _1}\big) \\
\hline
 s_2s_1 & \delt \frac{z_2}{z_1},\frac{\mu _3}{\mu _2}\big) \delt \frac{z_1}{z_3},h\big) \delt
   \frac{z_2}{z_3},h\big) \\
\hline
 s_1s_2s_1 & \delt \frac{z_1}{z_2},h\big) \delt \frac{z_1}{z_3},h\big) \delt \frac{z_2}{z_3},h\big)\\
\hline   
\end{array}
$$

\subsection{$G={\rm Sp}_2$} The Weil group is generated by two reflections:
$$\begin{array}{llllll}
s_1&\text{reflection in }& \alpha_1=(1,-1),&\text{the dual root }&\alpha_1^\vme=(1,-1),\\
s_2&\text{reflection in }& \alpha_2=(0,2), &\text{the dual root }&\alpha_2^\vme=(0,1).
\end{array}
$$
We have $s_1s_2s_1s_2=s_2s_1s_2s_1$. The full table of localized elliptic classes is given below:

$$\def\arraystretch{2}\begin{array}{|c|c|c|c|c|c|}\hline
_\sigma\!\!\!{\text{\large$\diagdown$}}\!\!^\om&~~\id~~&s_1&s_2&s_1s_2&s_2s_1\\ \hline\hline
\text{id}&1&\delt\frac{z_2}{z_1},\frac{\mu_2}{\mu_1}\big)&\delt\frac{1}{z_2^2},\frac{1}{\mu_2}\big)&\delt\frac{1}{z_2^2},\frac{1}{\mu_2}\big)\delt\frac{z_2}{z_1},\frac{1}{\mu_1\mu_2}\big)&\delt\frac{1}{z_2^2},\frac{1}{\mu_1}\big)\delt\frac{z_2}{z_1},\frac{\mu_2}{\mu_1}\big)\\ \hline
s_1&0&\delt\frac{z_1}{z_2},h\big)&0&\delt\frac{z_1}{z_2},h\big)\delt\frac{1}{z_1^2},\frac{1}{\mu_2}\big)&\delt\frac{z_1}{z_2},h\big)\delt\frac{1}{z_2^2},\frac{1}{\mu_1}\big)\\ \hline
s_2&0&0&\delt z_2^2,h\big)&\delt z_2^2,h\big)\delt\frac{z_2}{z_1},\frac{1}{\mu_1\mu_2}\big)&\delt z_2^2,h\big)\delt\frac{1}{z_1z_2},\frac{\mu_2}{\mu_1}\big)\\ \hline
s_1s_2&0&0&0&\delt z_1^2,h\big)\delt\frac{z_1}{z_2},h\big)&0\\ \hline
s_2s_1&0&0&0&0&\delt z_1z_2,h\big)\delt z_2^2,h\big)\\ \hline
s_1s_2s_1&0&0&0&0&0\\ \hline
s_2s_1s_2&0&0&0&0&0\\ \hline
s_1s_2s_1s_2&0&0&0&0&0\\ \hline
\end{array} $$

{
\tiny 
$$\arraycolsep=1.4pt\def\arraystretch{2}\begin{array}{|c|c|c|}\hline
_\sigma\!\!\!{\text{\large$\diagdown$}}\!\!^\om&s_1s_2s_1&s_2s_1s_2\\ \hline \hline
\text{id}&\delt\frac{z_1}{z_2},h\big)\delt\frac{z_2}{z_1},h\big)\delt\frac{1}{z_1^2},\frac{1}{\mu_1}\big)+\delt\frac{1}{z_2^2},\frac{1}{\mu_1}\big)\delt\frac{z_2}{z_1},\frac{1}{\mu_1\mu_2}\big)\delt\frac{z_2}{z_1},\frac{\mu_2}{\mu_1}\big)
&\delt\frac{1}{z_2^2},h\big)\delt z_2^2,h\big)\delt\frac{1}{z_1z_2},\frac{1}{\mu_1\mu_2}\big)+\delt\frac{1}{z_2^2},\frac{1}{\mu_1}\big)\delt\frac{1}{z_2^2},\frac{1}{\mu_2}\big)\delt\frac{z_2}{z_1},\frac{1}{\mu_1\mu_2}\big)
\\ \hline
s_1&\delt\frac{z_1}{z_2},h\big)\delt\frac{1}{z_1^2},\frac{1}{\mu_1}\big)\delt\frac{z_1}{z_2},\frac{\mu_2}{\mu_1}\big)+\delt\frac{z_1}{z_2},h\big)\delt\frac{1}{z_2^2},\frac{1}{\mu_1}\big)\delt\frac{z_2}{z_1},\frac{1}{\mu_1\mu_2}\big)
&\delt\frac{z_1}{z_2},h\big)\delt\frac{1}{z_1^2},\frac{1}{\mu_2}\big)\delt\frac{1}{z_2^2},\frac{1}{\mu_1}\big)
\\ \hline
s_2&\delt z_2^2,h\big)\delt\frac{1}{z_1z_2},\frac{\mu_2}{\mu_1}\big)\delt\frac{z_2}{z_1},\frac{1}{\mu_1\mu_2}\big)
&\delt z_2^2,h\big)\delt\frac{1}{z_2^2},\frac{1}{\mu_1}\big)\delt\frac{z_2}{z_1},\frac{1}{\mu_1\mu_2}\big)+\delt z_2^2,h\big)\delt\frac{1}{z_1z_2},\frac{1}{\mu_1\mu_2}\big)\delt z_2^2,\frac{1}{\mu_2}\big)
\\ \hline
s_1s_2&\delt z_1^2,h\big)\delt\frac{z_1}{z_2},h\big)\delt\frac{1}{z_1z_2},\frac{\mu_2}{\mu_1}\big)
&\delt z_1^2,h\big)\delt\frac{z_1}{z_2},h\big)\delt\frac{1}{z_2^2},\frac{1}{\mu_1}\big)
\\ \hline
s_2s_1&\delt z_1z_2,h\big)\delt z_2^2,h\big)\delt\frac{z_2}{z_1},\frac{1}{\mu_1\mu_2}\big)
&\delt z_1z_2,h\big)\delt z_2^2,h\big)\delt\frac{1}{z_1^2},\frac{1}{\mu_2}\big)
\\ \hline
s_1s_2s_1&\delt z_1^2,h\big)\delt\frac{z_1}{z_2},h\big)\delt z_1z_2,h\big)
&0
\\ \hline
s_2s_1s_2&0
&\delt z_1^2,h\big)\delt z_1z_2,h\big)\delt z_2^2,h\big)
\\ \hline
s_1s_2s_1s_2&0&0\\ \hline
\end{array} $$

$$\arraycolsep=1.4pt\def\arraystretch{2}\begin{array}{|c|c|}\hline 
_\sigma\!\!\!{\text{\large$\diagdown$}}\!\!^\om&s_1s_2s_1s_2=s_2s_1s_2s_1\\ \hline
\text{id}&\delt\frac{1}{z_1^2},\frac{1}{\mu_1}\big)\delt\frac{1}{z_2^2},\frac{1}{\mu_2}\big)\delt\frac{z_1}{z_2},h\big)\delt\frac{z_2}{z_1},h\big)+\delt\frac{1}{z_2^2},\frac{1}{\mu_1}\big)\delt\frac{1}{z_2^2},\frac{1}{\mu_2}\big)\delt\frac{z_2}{z_1},\frac{1}{\mu_1\mu_2}\big)\delt\frac{z_2}{z_1},\frac{\mu_2}{\mu_1}\big)+\delt\frac{1}{z_2^2},h\big)\delt\frac{1}{z_1z_2},\frac{1}{\mu_1\mu_2}\big)\delt\frac{z_2}{z_1},\frac{\mu_2}{\mu_1}\big)\delt z_2^2,h\big)\\ \hline
s_1&\delt\frac{1}{z_1^2},h\big)\delt z_1^2,h\big)\delt\frac{1}{z_1z_2},\frac{1}{\mu_1\mu_2}\big)\delt\frac{z_1}{z_2},h\big)+\delt\frac{1}{z_1^2},\frac{1}{\mu_1}\big)\delt\frac{1}{z_1^2},\frac{1}{\mu_2}\big)\delt\frac{z_1}{z_2},\frac{1}{\mu_1\mu_2}\big)\delt\frac{z_1}{z_2},h\big)+\delt\frac{1}{z_1^2},\frac{1}{\mu_2}\big)\delt\frac{1}{z_2^2},\frac{1}{\mu_1}\big)\delt\frac{z_2}{z_1},\frac{\mu_2}{\mu_1}\big)\delt\frac{z_1}{z_2},h\big)\\ \hline
s_2&\delt\frac{1}{z_1^2},\frac{1}{\mu_1}\big)\delt\frac{z_1}{z_2},h\big)\delt\frac{z_2}{z_1},h\big)\delt z_2^2,h\big)+\delt\frac{1}{z_2^2},\frac{1}{\mu_1}\big)\delt\frac{z_2}{z_1},\frac{1}{\mu_1\mu_2}\big)\delt\frac{z_2}{z_1},\frac{\mu_2}{\mu_1}\big)\delt z_2^2,h\big)+\delt\frac{1}{z_1z_2},\frac{1}{\mu_1\mu_2}\big)\delt\frac{z_2}{z_1},\frac{\mu_2}{\mu_1}\big)\delt z_2^2,\frac{1}{\mu_2}\big)\delt z_2^2,h\big)\\ \hline
s_1s_2&\delt z_1^2,h\big)\delt z_1^2,\frac{1}{\mu_2}\big)\delt\frac{1}{z_1z_2},\frac{1}{\mu_1\mu_2}\big)\delt\frac{z_1}{z_2},h\big)+\delt\frac{1}{z_1^2},\frac{1}{\mu_1}\big)\delt z_1^2,h\big)\delt\frac{z_1}{z_2},\frac{1}{\mu_1\mu_2}\big)\delt\frac{z_1}{z_2},h\big)+\delt z_1^2,h\big)\delt\frac{1}{z_2^2},\frac{1}{\mu_1}\big)\delt\frac{z_2}{z_1},\frac{\mu_2}{\mu_1}\big)\delt\frac{z_1}{z_2},h\big)\\ \hline
s_2s_1&\delt\frac{1}{z_1^2},\frac{1}{\mu_2}\big)\delt\frac{z_2}{z_1},\frac{\mu_2}{\mu_1}\big)\delt z_1z_2,h\big)\delt z_2^2,h\big)\\ \hline
s_1s_2s_1&\delt z_1^2,h\big)\delt\frac{1}{z_2^2},\frac{1}{\mu_2}\big)\delt\frac{z_1}{z_2},h\big)\delt z_1z_2,h\big)\\ \hline
s_2s_1s_2&\delt z_1^2,h\big)\delt\frac{z_2}{z_1},\frac{\mu_2}{\mu_1}\big)\delt z_1z_2,h\big)\delt z_2^2,h\big)\\ \hline
s_1s_2s_1s_2&\delt z_1^2,h\big)\delt\frac{z_1}{z_2},h\big)\delt z_1z_2,h\big)\delt z_2^2,h\big)\\ \hline
\end{array} $$
}


\begin{thebibliography}{AMSS2030}


\bibitem[AgOk16]{AO}
M. Aganagic, A. Okounkov, {Elliptic stable envelopes}, Preprint 2016, {\sf arXiv:1604.00423}

\bibitem[AlMi16]{AlMi} P. Aluffi, L. C. Mihalcea.
 Chern-Schwartz-MacPherson classes for Schubert cells in flag manifolds. Compos. Math. 152 (2016), no. 12, 2603--2625

\bibitem[AMSS17]{AMSS} 
P. Aluffi, L. C. Mihalcea, J. Sch{\"u}rmann, Ch. Su.
Shadows of characteristic cycles, Verma modules, and positivity of Chern-Schwartz-MacPherson classes of Schubert cells, arXiv:1709.08697 

\bibitem[AMSS19]{AMSS2}
P. Aluffi, L. C. Mihalcea, J. Sch\"urmann, Ch. Su. 
Motivic Chern classes of Schubert cells, Hecke algebras, and applications to Casselman's problem, 
arXiv:1902.10101

\bibitem[BGG73]{BGG} 
I.~N. Bern{\v{s}}te{\u\i}n, I.~M. Gel'fand, and S.~I. Gel'fand.
Schubert cells, and the cohomology of the spaces {$G/P$}.
Uspehi Mat. Nauk, 28(3(171)):3--26, 1973

\bibitem[BL00]{BoLi0}
L. Borisov, A. Libgober.
Elliptic genera of toric varieties and applications to mirror symmetry.
Inv. Math. (2000) 140: 453

\bibitem[BL03]{BoLi1}
L. Borisov, A. Libgober.
Elliptic genera of singular varieties. 
Duke Math. J. 116 (2003), no.~2, 319--351

\bibitem[BL05]{BoLi2}
L. Borisov, A. Libgober.
Mc{K}ay correspondence for elliptic genera.
Ann. of Math. (2), 161(3):1521--1569, 2005

\bibitem[BSY10]{BSY}
J.-P. Brasselet, J. Sch{\"u}rmann, Sh. Yokura.
Hirzebruch classes and motivic {C}hern classes for singular spaces.
J. Topol. Anal., 2(1):1--55, 2010

\bibitem[BE90]{BE}
P. Bressler, S. Evens. 
The Schubert calculus, braid relations, and generalized cohomology.
Trans. Amer. Math. Soc. 317 (1990), no. 2, 799--811


\bibitem[Br05]{BrionFlag} 
M. Brion.
Lectures on the geometry of flag varieties. Topics in cohomological studies of algebraic varieties, 33--85,
Trends Math., Birkh\"auser, Basel, 2005

\bibitem[BrKu05]{BrKu} 
M. Brion, S. Kumar, 
Frobenius splitting methods in geometry and representation theory.
Progress in Mathematics, 231. Birkh\"auser Boston, Inc., Boston, MA, 2005

\bibitem[Che58]{Che} 
C. Chevalley, 
Sur les d{\'e}compositions cellulaires des espaces $G/B$.
Algebraic groups and their generalizations: classical methods (University Park, PA, 1991)
Proc. Sympos. Pure Math. 56, 1--23
   
\bibitem[Cha85]{Cha}
K.~Chandrasekharan.
Elliptic functions, volume 281 of {\em Grundlehren der
  Mathematischen Wissenschaften}. 
Springer-Verlag, Berlin, 1985

\bibitem[ChGi97]{ChGi}
N.~Chriss and V.~Ginzburg.
Representation Theory and Complex Geometry.
Modern Birkh{\"a}user Classics. Birkh{\"a}user Boston, 1997, 2009

\bibitem[De74]{Dem} 
M.~Demazure,
D{\'e}singularisation des vari{\'e}t{\'e}s de Schubert g{\'e}n{\'e}ralis{\'e}es. 
Collection of articles dedicated to Henri Cartan on the occasion of his 70th birthday, I. Ann. Sci. {\'E}cole Norm. Sup. (4) 7 (1974), 53--88

\bibitem[DBWe16]{DBW} 
M. Donten-Bury, A. Weber. 
Equivariant Hirzebruch classes and Molien series of quotient singularities. 
Transform. Groups 23 (2018), no. 3, 671--705

\bibitem[F73]{Fay}
J. D. Fay: Theta functions on Riemann surfaces. 
Springer LNM 352, 1973

\bibitem[FR18]{FR}
L. Feh\'er, R. Rim\'anyi. 
Chern-Schwartz-MacPherson classes of degeneracy loci. 
Geometry and Topology 22 (2018) 3575--3622

\bibitem[FRW18]{FRW} 
L. M. Feh\'er, R. Rim\'anyi, A. Weber.
Motivic Chern classes and K-theoretic stable envelopes,
arXiv:1802.01503 
 
\bibitem[FRV07]{FRV}
G. Felder, R. Rim\'anyi, A. Varchenko. 
Poincar\'e-Birkhoff-Witt expansions of the canonical elliptic differential form, in ``Quantum Groups'' (eds. P. Etingof, S. Gelaki, S. Shnider), Contemp. Math. 433 (2007), 191--208

\bibitem[FRV18]{FRV18}
G. Felder, R. Rim\'anyi, A. Varchenko. 
Elliptic dynamical quantum groups and equivariant elliptic cohomology, SIGMA (Symmetry, Integrability and Geometry: Methods and Applications) 14 (2018), 132, 41 pages

\bibitem[FRW19]{FRW2} 
L. M. Feh\'er, R. Rim\'anyi, A. Weber.
Characteristic classes of orbit stratifications, the axiomatic approach. To appear in the proceedings of the Schubert Calculus conference in Guangzhou 2017, Springer. arXiv:1811.11467 

\bibitem[FF16]{FoFu}
A. Fomenko and D. Fuchs.
\newblock {\em Homotopical topology}. 
{GTM 273}, Springer, 2nd ed., 2016

\bibitem[Ga14]{Ganter}
N. Ganter.
\newblock The elliptic {W}eyl character formula.
\newblock {\em Compos. Math.}, 150(7):1196--1234, 2014

\bibitem[GaT-L17]{GL} 
S. Gautam, V. Toledano Laredo.
Elliptic quantum groups and their finite-dimensional representations. 
arXiv:1707.06469


\bibitem[GKV95]{GKV}
V. Ginzburg, M. Kapranov, and Eric Vasserot.
\newblock Elliptic algebras and equivariant elliptic cohomology, 1995.
\newblock arXiv:q-alg/9505012


\bibitem[GrKu08]{GrKu} 
W. Graham, S. Kumar.
On positivity in T-equivariant K-theory of flag varieties. 
Int. Math. Res. Not. IMRN 2008, Art. ID rnn 093, 43 pp. 

\bibitem[Gro94]{Grojnowski} 
I. Grojnowski. 
Delocalised equivariant elliptic cohomology. Preprint 1994, \newblock In {\em  Elliptic cohomology}, 114-121,
London Math. Soc. Lecture Note Ser., 342, Cambridge Univ. Press, Cambridge, 2007

\bibitem[Hi88]{Hirzebruch} 
F. Hirzebruch,
Elliptic genera of level N for complex manifolds. Differential geometrical methods in theoretical physics (Como, 1987), 37-63,
NATO Adv. Sci. Inst. Ser. C Math. Phys. Sci., 250, Kluwer Acad. Publ., Dordrecht, 1988

\bibitem[HBJ92]{HirzebruchBook} F. Hirzebruch, T. Berger, R. Jung. Manifolds and modular forms.  Aspects of Mathematics, E20. Friedr. Vieweg \& Sohn, Braunschweig, 1992

\bibitem[H\"o91]{Hohn} 
G. H\"ohn.
Komplexe elliptische Geschlechter und $S^1$-aequivariante Kobordismustheorie. 
    Thesis 1991. arXiv:math/0405232 

\bibitem[Kem76]{Kem} 
G. R. Kempf.
Linear systems on homogeneous spaces.
Ann. of Math. (2) 103 (1976), no. 3, 557--591

\bibitem[Ko17]{konno}
H. Konno. 
Elliptic weight functions and elliptic q-KZ equation.
Journal of Integrable Systems, Vol.~2, Issue 1, 2017

\bibitem[KuSh14]{KS} 
S. Kumar, K. Schwede.
Richardson varieties have Kawamata log terminal singularities.
IMRN 2014, no. 3, 842--864

\bibitem[KR03]{KR}
I. Rosu. 
Equivariant K-theory and Equivariant Cohomology (with appendix by A. Knutson and I. Rosu). 
Math. Z., Vol. 243, Issue 3, pp 423--448 (2003)

\bibitem[Kr90]{Krichever} 
I. M. Krichever. 
Generalized elliptic genera and Baker-Akhiezer functions. Mat. Zametki 47 (1990), no. 2, 34-45, 158; translation in Math. Notes 47 (1990), no. 1-2, 132-142 

\bibitem[La88]{Landweber}
P. S. Landweber.
\newblock Elliptic genera: an introductory overview.
\newblock In {\em Elliptic curves and modular forms in algebraic topology
  ({P}rinceton, {NJ}, 1986)}, volume 1326 of {\em Lecture Notes in Math.},
  pages 1--10. Springer, Berlin, 1988

\bibitem[LZ15]{LZ}
C. Lenart, K. Zainoulline. 
A Schubert basis in equivariant elliptic cohomology. New York Journal of Mathematics. 23. 2015

\bibitem[Li18]{Li} 
A. Libgober.
Elliptic genus of singular algebraic varieties and quotients.
J. Phys. A 51 (2018), no. 7, 073001, 33 pp. 

\bibitem[Lu85]{Lusztig} 
G. Lusztig. 
Equivariant K-theory and representations of Hecke algebras. 
Proc. AMS, 94(2):337--342, 1985

\bibitem[MaOk12]{MO}
D. Maulik, A. Okounkov. 
{Quantum Groups and Quantum Cohomology},
Ast\'erisque 408, Soc. Math. de France 2019

\bibitem[OSWW17]{OSWW}  
G. Occhetta, L. E. Sol{\'a} Conde, K. Watanabe, J. A. Wi{\'s}niewski.
Fano manifolds whose elementary contractions are smooth P1-fibrations: a geometric characterization of flag varieties. 
Ann. Sc. Norm. Super. Pisa Cl. Sci. (5) 17 (2017), no. 2, 573--607

\bibitem[Oc87]{Ochanine} 
S. Ochanine. Sur les genres multiplicatifs d\'efinis par des int\'egrales elliptiques. Topology 26 (1987), 143-151

\bibitem[Oh06]{Oh} 
T. Ohmoto. 
Equivariant Chern classes of singular algebraic varieties with group actions.  
Math. Proc. Cambridge Phil. Soc. 140 (2006), 115--134

\bibitem[Oko17]{O} 
A. Okounkov.
Lectures on K-theoretic computations in enumerative geometry. 
Geometry of moduli spaces and representation theory, 251--380, IAS/Park City Math. Ser., 24, AMS 2017

\bibitem[Ram85]{Ram1} 
A. Ramanathan.
Schubert varieties are arithmetically Cohen-Macaulay.
Invent. math. 80, 283--294 (1985)

\bibitem[Ram87]{Ram2} 
A. Ramanathan.
Equations defining Schubert varieties and Frobenius splittings of diagonals.
Publ. math. de l\'{}I.H.\'E.S., tome 65 (1987), p. 61--90

\bibitem[RSVZ]{Sm2} 
R. Rim\'anyi, A. Smirnov, A. Varchenko, Z. Zhou.
Three dimensional mirror self-symmetry of the cotangent bundle of the full flag variety. SIGMA 15 (2019), 093, 22 pages

\bibitem[RTV17]{RTV}
R. Rim{\'a}nyi, V. Tarasov, A. Varchenko.
Elliptic and K-theoretic stable envelopes and Newton polytopes,
Selecta Math. (2019) 25:16, 43pp

\bibitem[SchY07]{SchYo} 
J. Sch\"urmann, Sh. Yokura. 
A survey of characteristic classes of singular spaces,  
Singularity theory, 865--952, World Sci. Publ., Hackensack, NJ, 2007. 

\bibitem[SZZ17]{SZZ}
Ch. Su, G. Zhao, Ch. Zhong. 
On the K-theory stable bases of the Springer resolution.  
arXiv:\-1708.08013


\bibitem[TV97]{TV}
V. Tarasov, A. Varchenko. 
Geometry of q-hypergeometric functions as a bridge between Yangians and quantum affine algebras. 
Invent. Math. 128 (1997), no. 3, 501--588

\bibitem[To00]{Totaro}
B. Totaro.
\newblock Chern numbers for singular varieties and elliptic homology.
\newblock {Ann. of Math. (2)}, 151(2):757--791, 2000.

\bibitem[Wae08]{Wae}
R. Waelder.
Equivariant elliptic genera and local {M}c{K}ay correspondences.
Asian J. Math., 12(2):251--284, 2008.

\bibitem[We12]{WeCSM} 
A. Weber.  
Equivariant Chern classes and localization theorem.
Journal of Singularities, Vol. 5 (2012), 153--176

\bibitem[We98]{We} 
H. Weber,
Lehrbuch der Algebra. Zweite Auflage. Dritter Band: Elliptische Funktionen und algebraische Zahlen.
Braunschweig: F. Vieweg Sohn. (1898, 1908).

\bibitem[Wi88]{Witten}
E. Witten. The index of the Dirac operator in loop space. Elliptic curves and modular forms in algebraic topology (Princeton, NJ, 1986), 161--181, LNM 1326, Springer 1988
\end{thebibliography}
\end{document}